\documentclass[11pt,english]{amsart}
\usepackage[T1]{fontenc}
\usepackage[latin9]{inputenc}
\usepackage{geometry}
\usepackage{color}
\usepackage{babel}
\usepackage{amstext}
\usepackage{amsthm}
\usepackage{amssymb}
\makeindex
\usepackage[unicode=true,
 bookmarks=true,bookmarksnumbered=true,bookmarksopen=true,bookmarksopenlevel=1,
 breaklinks=true,pdfborder={0 0 0},pdfborderstyle={},backref=page,colorlinks=true]
 {hyperref}
\hypersetup{pdftitle={Polynomial Maps and Polynomial Sequences in Groups},
 pdfauthor={Yaqing Hu},
 pdfsubject={11C08, 20M14, 20F18},
 pdfkeywords={Key words and phrases: polynomial maps, polynomial sequences, commutative semigroups, locally nilpotent groups}}

\makeatletter
\numberwithin{equation}{section}
\numberwithin{figure}{section}
\theoremstyle{plain}
\newtheorem*{question*}{\protect\questionname}
\theoremstyle{plain}
\newtheorem{thm}{\protect\theoremname}
\theoremstyle{remark}
\newtheorem*{rem*}{\protect\remarkname}
\theoremstyle{plain}
\newtheorem{lem}{\protect\lemmaname}
\theoremstyle{definition}
\newtheorem{defn}{\protect\definitionname}
\theoremstyle{plain}
\newtheorem{prop}{\protect\propositionname}
\theoremstyle{plain}
\newtheorem{cor}{\protect\corollaryname}
\theoremstyle{definition}
 \newtheorem{example}{\protect\examplename}
\theoremstyle{definition}
\newtheorem*{example*}{\protect\examplename}
\theoremstyle{remark}
\newtheorem{rem}{\protect\remarkname}
\theoremstyle{definition}
\newtheorem*{defn*}{\protect\definitionname}
\theoremstyle{plain}
\newtheorem*{thm*}{\protect\theoremname}

\DeclareMathOperator{\lcm}{lcm}
\DeclareMathOperator{\Tor}{Tor}
\DeclareMathOperator{\Aut}{Aut}
\DeclareMathOperator{\Dih}{Dih}
\DeclareMathOperator{\rank}{rank}
\usepackage{pgfplots}
\pgfplotsset{compat=1.11}
\usepackage{xcolor}
\usepackage{tikz}
\usepackage{tikz-cd}
\usepackage{tikz-3dplot}
\usetikzlibrary{%
  positioning,%
  patterns,%
  shapes,%
  matrix,%
  calc,%
  arrows%
}
\allowdisplaybreaks

\makeatother

\providecommand{\corollaryname}{Corollary}
\providecommand{\definitionname}{Definition}
\providecommand{\examplename}{Example}
\providecommand{\lemmaname}{Lemma}
\providecommand{\propositionname}{Proposition}
\providecommand{\questionname}{Question}
\providecommand{\remarkname}{Remark}
\providecommand{\theoremname}{Theorem}

\begin{document}
\title{Polynomial Maps and Polynomial Sequences in Groups}
\author{Ya-Qing Hu}
\address{Morningside Center of Mathematics\\
Chinese Academy of Sciences\\
No. 55, Zhongguancun East Road\\
Haidian District, Beijing 100190}
\email{\href{mailto:yaqinghu@amss.ac.cn}{yaqinghu@amss.ac.cn}}
\thanks{This work is partially supported by NSF Grant DMS-1401419. }
\begin{abstract}
 This paper develops a theory of polynomial maps from commutative
semigroups to arbitrary groups and proves that it has desirable formal
properties when the target group is locally nilpotent. We will apply
this theory to solve Waring's problem for general discrete Heisenberg
groups in a sequel to this paper. 
\end{abstract}

\date{\today}

\maketitle
\tableofcontents{}

\section{Introduction}

\subsection{Motivation}

 The motivation of this work is the following question of Michael
Larsen:  
\begin{question*}
Find good notions of ``polynomial sequence'' and ``generalized cone''
so that if $G$ is a finitely generated nilpotent group and $g_{0},g_{1},g_{2},\ldots$
is a polynomial sequence in $G$ such that no coset of any infinite
index subgroup of $G$ contains the whole sequence, then there exists
a positive integer $M$, a generalized cone $C\subset G$, and a subgroup
$H$ of finite index in $G$ such that every element of $C\cap H$
is a product of $M$ elements of the sequence. 
\end{question*}
A polynomial sequence in $\mathbb{Z}$ should be given by a polynomial
$\mathbb{N}_{0}\to\mathbb{Z}$ and a generalized cone should be the
set of all integers in an unbounded open interval. Thus, the classical
Waring's problem with integer-valued polynomials of degree $\ge2$
solved by Kamke \cite{Kamke1921} should be a trivial consequence
of the above question for $G=\mathbb{Z}$. The goal is to find a uniform
notion of polynomial maps from nonempty semigroups to groups that
would allow us to define both polynomial sequences $\mathbb{N}_{0}\to\mathcal{U}_{n}(\mathbb{Z})$
and generalized cones as the image of continuous polynomial maps $\mathbb{R}_{\ge0}^{N}\to\mathcal{U}_{n}(\mathbb{R})$
with nonempty interiors, where $\mathcal{U}_{n}(\mathbb{R})$ is the
group of $n\times n$ unipotent matrices over $\mathbb{R}$. 

\subsection{Work of Leibman}

Leibman developed a theory of polynomial sequences $\mathbb{Z}\to G$
in any group $G$ \cite{Leibman1998} and polynomial mappings $G\to F$
between two groups \cite{Leibman2002}. Let $G=G_{1}\supseteq G_{2}\supseteq G_{2}\supseteq\cdots$
be the lower central series of $G$. In \cite{Leibman1998}, he defines
a difference operator $Dg(n)=g(n)^{-1}g(n+1)$ on the sequence $g:\mathbb{Z}\to G$
and calls $g$ a polynomial if for any $k$ there exists $d$ such
that the sequence obtained from $g$ by applying this difference operator
$d$ times takes its values in $G_{k}$, i.e., $D^{d}g(n)\in G_{k}$
for all $n\in\mathbb{Z}$. Then, he introduces the notion of the degree
of a polynomial sequence and proves that polynomial sequences of degrees
not exceeding a fixed superadditive sequence form a group with group
law defined elementwise. 

In \cite{Leibman2002}, given any $h\in G$, he defines the left (resp.
right) $h$-derivative of a mapping $\varphi:G\to F$ between two
groups by $D_{h}^{L}\varphi(g)=\varphi(hg)\varphi(g)^{-1}$ (resp.
$D_{h}^{R}\varphi(g)=\varphi(g)^{-1}\varphi(gh)$), and calls $\varphi$
a left-polynomial (resp. right-polynomial) mapping of degree $\le d$
if for all $h_{1},\ldots,h_{d+1}\in G$, $D_{h_{1}}^{L}\cdots D_{h_{d+1}}^{L}\varphi\equiv1_{F}$
(resp. $D_{h_{1}}^{R}\cdots D_{h_{d+1}}^{R}\varphi\equiv1_{F})$).
Then, he proves that if $F$ is nilpotent, right-polynomial mappings
$G\to F$ form a group with group law defined elementwise,  and $\varphi:G\to F$
is a right-polynomial if and only if it is a left-polynomial.  However,
the degree of polynomial map $f$ as a right-polynomial is not necessarily
the same as the degree of $f$ as a left-polynomial. 

\subsection{Our generalization}

 To meet our own needs, we modify and generalize Leibman's theory.
In particular, a polynomial of degree $d$ should be killed by any
sequence of $d+1$ difference operators, left, right, or a combination
of the two. The difficulty which this definition is intended to meet
is the unavailability of inverses in general semigroups.  The most
important quantity of a polynomial map is its lc-degree, which is
a vector formed by the degree of the induced polynomial map modulo
lower central series of the target group and conveys more information
than the degree. A polynomial sequence in $G$ is given by a polynomial
map from $\mathbb{N}_{0}$ to $G$. A generalized cone in a path-connected
nilpotent Lie group $N$ is given as the image of a continuous polynomial
map $\mathbb{R}_{\ge0}^{n}\to N$, which is assumed to have non-empty
interior. It turns out that a generalized cone in $\mathbb{R}$ is
just an unbounded interval.  Given a homomorphism from a nilpotent
group $G$ to a path-connected nilpotent Lie group $N$, we can pull
back a generalized cone in $N$ to obtain a generalized cone in $G$.

\subsection{Main results}

Just as a polynomial $\mathbb{R}\to\mathbb{R}$ of degree $d$ is
in general determined by $d+1$ polynomial values, a polynomial map
is uniquely determined by certain special values and $\langle f(S)\rangle$
is finitely generated if $S$ is finitely generated. The set of polynomial
maps from $S$ to $G$ is invariant under conjugations in $G$, translations
in both $S$ and $G$, and taking elementwise inverse. In particular,
the fact that the elementwise inverse $f^{-1}:S\to G$ is also a polynomial
map of the same degree resembles the fact that the additive inverse
$-f$ of a polynomial $f\in R[x]$ is a polynomial of the same degree
with $f$. (Leibman's left or right polynomial does not have such
a nice property.) 

Elementwise product of two polynomial maps $S\to G$ may not be a
polynomial map. The simplest example might be provided by the multiplicative
functions $f_{1}(n)=x^{n}$ and $f_{2}(n)=y^{n}$ from $\mathbb{N}$
to the free group $F_{2}$ generated by two generators $x,y$. Even
in metabelian (let alone solvable) groups, the situation is still
unpleasant; cf. Example \ref{exa:Fibonacci} and the remark following
it.  The first main result is about elementwise product of polynomial
maps: 
\begin{thm}
\label{thm:product of two polynomial maps} Let $S$ be any nonempty
commutative semigroup, $G$ be any group and $f,f':S\rightarrow G$
be polynomial maps of degree $\le d$ and respectively $\le d'$.
If the subgroup $\langle f,f'\rangle$ generated by $f(S)$ and $f'(S)$
is nilpotent,  then the (elementwise) product 
\[
ff':S\to G;\quad t\mapsto f(t)f'(t)
\]
is a polynomial map.
\end{thm}
It follows that if $G$ is nilpotent of class $n$, then all polynomial
maps from $S$ to $G$ form a nilpotent group of class $n$, cf. Corollary
\ref{cor:group of polynomial maps}. Leibman proves something similar:
if $F$ is nilpotent, right-polynomial mappings $G\to F$ form a group,
cf. \cite[Theorem 3.2]{Leibman2002}. Both of these proofs are elementary
 and essentially done by induction on the (lc-)degree of the polynomial
map and the nilpotency class of the target group. Moreover, Leibman
proves that if $F$ is nilpotent of class $c$, then $\varphi:G\to F$
is a right-polynomial if and only if $\varphi$ is a left-polynomial,
cf. \cite[Proposition 3.16]{Leibman2002}. In fact, he proves if $\varphi$
is a right-polynomial  of degree $\le d$, then $\varphi$ is a left-polynomial
 of degree $\le dc^{2}$. But if the target group is locally nilpotent
but not nilpotent, this statement does not necessarily hold; see Example
\ref{exa:polynomials maps in locally nilpotent groups} and Remark
\ref{rem:left but not right polynomial} for more detail. It is not
the purpose of this paper to distinguish the slight difference between
polynomial maps and left or right polynomial mappings. 

 Polynomial maps $\mathbb{R}_{\ge0}\to\mathbb{R}$ in our sense are
not necessarily given by polynomials in the usual sense. Discontinuous
additive functions $\mathbb{R}_{\ge0}\to\mathbb{R}$ provide such
pathological examples of polynomial maps of degree $1$; cf. Remark
\ref{rem:Hamel basis}. This can be avoided if the continuity is required.
Indeed, Theorem \ref{thm:polynomial in N var} shows that every continuous
polynomial map $f:\mathbb{R}_{\ge0}^{N}\rightarrow\mathbb{R}$ is
the usual polynomial. Theorem \ref{thm:unitriangular matrix of polynomials in N var}
shows that for $1\le i<j\le n$ each entry $f_{i,j}$ of the matrix
form of any continuous polynomial map $f:\mathbb{R}_{\ge0}^{N}\to\mathcal{U}_{n}(\mathbb{R})$
is a polynomial $f_{i,j}:\mathbb{R}_{\ge0}^{N}\to\mathbb{R}$. 

 The most important quantity of a polynomial map is its (lc-)degree.
So we are very interested in the lower and upper bounds of the (lc-)degree,
in particular, of polynomial maps of the form $\mathbb{R}_{\ge0}^{N}\to\mathcal{U}_{n}(\mathbb{R})$.
Theorems \ref{thm:bound degree N} and \ref{thm:bound lc-degree N}
gives lower and upper bounds of the (lc-)degree of $f:\mathbb{R}_{\ge0}^{N}\to\mathcal{U}_{n}(\mathbb{R})$
via the degree of $f_{i,j}$. 

 Theorem \ref{thm:polynomial sequence repeats constant} states that
a nonconstant polynomial sequence $g:\mathbb{N}_{0}\to G$ in a finitely
generated torsion-free nilpotent group $G$ cannot repeat the same
value infinitely many times. Theorem \ref{thm:infinite subsequence generates finite index subgroup}
states that every infinite subsequence (not necessarily corresponding
to any arithmetic progression) in any nilpotent group generates a
finite index subgroup of the group generated by the whole sequence. 

 Denote the direct sum of $N$ copies of a commutative semigroup
$S$ by $S^{N}$ and the set of all polynomial maps from $S^{N}$
to a group $G$ by $G_{p}^{S^{N}}$, on which the symmetric group
$S_{N}$ naturally acts. Then, we call a polynomial map $f:S^{N}\to G$
symmetric with respect to this $S_{N}$-action, if it is invariant
under this action. The strategy that we call the \emph{iterated symmetrization}
enables us to prove Theorem \ref{thm:symmetrization in G}, which
guarantees that any polynomial map $f:S^{N}\to G$, where $G$ is
nilpotent of class $n$, can be turned into a symmetric polynomial
map $\tilde{f}=\sigma_{1}(f)\sigma_{2}(f)\cdots\sigma_{M}(f)$, where
$\sigma_{1},\sigma_{2},\ldots,\sigma_{M}\in S_{N}$. Results in this
section will lay the foundation for our work on Waring's problem in
discrete Heisenberg groups. 

\subsection*{Organization}

In Section \ref{sec:Polynomial Maps}, we generalize the usual polynomials
to polynomial maps from a nonempty commutative semigroup $S$ to a
group $G$. Proposition \ref{prop:polynomial map of degree 1} and
Corollary \ref{cor:f-translation} characterize polynomial maps of
degree $1$.  Two possible ways to construct induced polynomial maps
via homomorphisms of semigroups or groups are given in Propositions
\ref{prop:homomorphism of commutative semigroups} and \ref{prop:homomorphism of groups}
respectively. Then, we give an elementary proof of Theorem \ref{thm:product of two polynomial maps}.
We also prove that polynomial maps from a commutative semigroup to
a nilpotent group with lc-degree bounded by a fixed superadditive
vector form a nilpotent subgroup; cf. Corollary \ref{cor:lc-degree <=00003D a superadditive vector}.
Apart from elementwise product, we also talk about ordered product
$f\odot f':S\times S'\to G$ of two polynomial maps $f:S\to G$ and
$f':S'\to G$ given by $f\odot f'(s,s')=f(s)f'(s')$ and prove that
the ordered product $\bigodot_{i=1}^{k}f$ of polynomial map $f:S\to G$
with lc-degree bounded by a fixed superadditive vector is a polynomial
map with lc-degree bounded by the same superadditive vector. 

Section \ref{sec:Continuous Polynomial Maps} is devoted to a characterization
of (continuous) polynomial maps in several variables, such as $f:\mathbb{R}_{\ge0}^{N}\rightarrow\mathbb{R}$
and $f:\mathbb{R}_{\ge0}^{N}\rightarrow\mathcal{U}_{n}(\mathbb{R})$. 

Section \ref{sec:Degree} provides estimations of lower and upper
bounds for the (lc-)degree of polynomial maps $f:\mathbb{R}_{\ge0}^{N}\rightarrow\mathcal{U}_{n}(\mathbb{R})$. 

Section \ref{sec:Polynomial Sequences} consists of basic results
about polynomial sequences and subsequences in a group $G$, which
are polynomial maps $\mathbb{N}_{0}\to G$. 

Section \ref{sec:Symmetric Polynomial Maps} proves some technical
results about symmetric polynomial maps with the help of some $1$-cocycles
of non-abelian group cohomology. 

Section \ref{sec:Polynomial Sets} introduces the concept of polynomial
sets in nilpotent groups and finds a proper polynomial set inside
any Kamke domain, an open subset in $\mathbb{R}^{n}$ played an important
role in Kamke's work and our future work. 

\subsection*{Acknowledgment}

I thank my advisor Michael Larsen for the guidance and many helpful
discussions through this work, in particular, for bringing this problem
to my attention. 

\section{Polynomial Maps \label{sec:Polynomial Maps}}

 A semigroup $S$ is a set $S$ together with a binary operation
$\cdot:S\times S\rightarrow S$ that satisfies the associative property:
$(a\cdot b)\cdot c=a\cdot(b\cdot c)$, $\forall a,b,c\in S$. A monoid
$S$ is a semigroup with the identity $e\in S$ such that $e\cdot a=a\cdot e=a$,
$\forall a\in S.$ If the binary operation is commutative, i.e., $a\cdot b=b\cdot a$,
$\forall a,b\in S$, then the semigroup (or monoid) $S$ is called
commutative, and the binary operation is denoted by $+$ and the identity
is denoted by $0$. 
\begin{rem*}
Vacuously, the empty set with the empty function as the binary operation
forms an empty semigroup. All semigroups mentioned in this paper are
assumed to be nonempty. 
\end{rem*}
 The rank of a semigroup $S$ is the smallest cardinality of a generating
set for the semigroup, i.e., 
\[
\rank(S)=\min\{|X|:X\subseteq S,\ \langle X\rangle=S\text{ or }\langle X\rangle=S\setminus\{\text{the unity element}\}\}.
\]
Thus, the rank of a finitely generated semigroup is the minimal number
of elements generating this semigroup either by themselves or after
the addition of the unity element.  A commutative semigroup $(S,+)$
is called divisible, if for any $s\in S$ and any $n\in\mathbb{N}$,
there exists $t\in S$ such that $nt=s$, and is called uniquely divisible,
if $t$ is unique. 

We adopt the following conventions on commutators and conjugations
in any group $G$: The commutator of the elements $x,y$ in a group
$G$ is defined by $[x,y]:=xyx^{-1}y^{-1}$, the $y$-conjugate of
$x$ in $G$ by $x^{y}:=yxy^{-1}$, the $n$-fold left-commutator
of $x_{1},x_{2},\ldots,x_{n}$ in $G$ by 
\[
[x_{1},x_{2},\ldots,x_{n}]:=[[\ldots[x_{1},x_{2}],\ldots,x_{n-1}],x_{n}],
\]
and the $1$-fold left-commutator of $x$ simply by $[x]:=x$. 

The following commutator identities will greatly facilitate calculations
related to commutators. 
\begin{lem}
\label{lem:commutator identities} Let $G$ be a group and $x,y,z,x_{1},\cdots,x_{n}\in G$.
Then, the following identities hold:
\begin{enumerate}
\item $x^{y}=[y,x]x$;
\item \label{enu:=00005Bx,y=00005D^-1} $[x,y]^{-1}=[y,x]$;
\item $[x^{-1},y]=[x^{-1},[y,x]][y,x]=[y,x]^{x^{-1}}$;
\item \label{enu:=00005Bx,yz=00005D} $[x,yz]=[x,y][y,[x,z]][x,z]=[x,y][x,z]^{y}$;
\item \label{enu:=00005Bxy,z=00005D} $[xy,z]=[x,[y,z]][y,z][x,z]=[y,z]^{x}[x,z]$;
\item $[x_{1},\ldots,x_{n}]^{z}=[x_{1}^{z},\ldots,x_{n}^{z}]$; 
\item \label{enu:Hall-Witt} $[x^{-1},y,z]^{x}[z^{-1},x,y]^{z}[y^{-1},z,x]^{y}=1$
and $[y,x,z^{x}][x,z,y^{z}][z,y,x^{y}]=1$. 
\end{enumerate}
The Identity (\ref{enu:Hall-Witt}) is also known as the Hall-Witt
identity. 
\end{lem}
\begin{proof}
The verification of these statements is routine. 
\end{proof}
Let $X$ and $Y$ be subsets of a group $G$. The commutator subgroup
of $X$ and $Y$ is defined to be $[X,Y]:=\langle[x,y]\mid x\in X,y\in Y\rangle$.
In particular, the derived or commutator subgroup of $G$ is defined
to be $G^{(1)}=G'=[G,G]$. Identity (\ref{enu:=00005Bx,y=00005D^-1})
implies that the commutator subgroup is symmetric: $[X,Y]=[Y,X]$.

\begin{defn}
\label{def:polynomial map} We say that a map $f:S\to G$ is a polynomial
map of degree $-\infty$ \footnote{To be compatible with the definition of the lc-degree and superadditive
vectors later, $-\infty$ turns out to be a better choice for the
degree of the zero map than $-1$. } if $f$ maps $S$ to the identity $1_{G}$ of $G$, and $f$ is a
polynomial map of degree $0$ if it is a constant but not the identity.
Inductively, we say that $f$ is a polynomial map of degree $\le d+1$,
if for all $s\in S$ the following left and right forward finite differences
\[
L_{s}(f):S\rightarrow G;\ t\mapsto f(s+t)f(t)^{-1}\qquad R_{s}(f):S\rightarrow G;\ t\mapsto f(t)^{-1}f(s+t)
\]
are polynomial maps of degree $\le d$.

We call the minimal $d$ with this property the degree of the polynomial
map.  We call $L$ (resp. $R$) the left (resp. right) difference
operator. If $G$ is abelian, then there is no need to distinguish
$L$ from $R$, so $D$ is used to denote either one of them. 
\end{defn}
If $f:S\to G$ has degree $\le0$, then we may abuse notations and
simply denote its image by $f$ and thus any element $g\in G$ is
also viewed as a constant polynomial map $g:S\rightarrow G$. Let
$\mathbb{Z}_{*}=\mathbb{N}_{0}\cup\{-\infty\}$ and adopt the following
convention
\[
-\infty<n\text{ and }-\infty+n=-\infty=(-\infty)+(-\infty),\quad\forall n\in\mathbb{Z}
\]
to extend the addition in $\mathbb{N}_{0}$ to $\mathbb{Z}_{*}$,
and the following convention
\[
a-b=\begin{cases}
a-b, & \text{if }a\ge b,\\
-\infty, & \text{if }a<b,
\end{cases}\quad\forall a\in\mathbb{Z}_{*},\forall b\in\mathbb{N}_{0}
\]
to partially extend the subtraction in $\mathbb{N}_{0}$ to $\mathbb{Z}_{*}$
and leave $a-(-\infty)$ undefined.
\begin{rem*}
With this convention, the definition of polynomial maps can be summarized
as follows: $f:S\to G$ is a polynomial map of degree $\le d$, if
for any $s_{1},s_{2},\ldots,s_{d+1}\in S$, 
\[
D_{s_{1}}D_{s_{2}}\cdots D_{s_{d+1}}f\equiv1_{G},
\]
where each $D$ is arbitrarily taken to be $L$ or $R$. 

Any nonconstant multiplicative function from $S$ to $G$ is a polynomial
map of degree $1$. Moreover, 
\end{rem*}
\begin{prop}
\label{prop:polynomial map of degree 1} If $S$ is a commutative
monoid, then any polynomial map $f:S\to G$ of degree $1$ is a nonconstant
affine multiplicative function, i.e., a multiplicative function multiplied
by a constant in $G$ either on the left or on the right.
\end{prop}
\begin{proof}
A nonconstant map $f:S\rightarrow G$ is a polynomial map of degree
$\le1$, if for each $s\in S$, 
\[
L_{s}(f):S\rightarrow G;\ t\mapsto f(s+t)f(t)^{-1}\qquad R_{s}(f):S\rightarrow G;\ t\mapsto f(t)^{-1}f(s+t)
\]
are polynomial maps of degree $\le0$. Thus, we have $f(s+t)=l_{s}f(t)=f(t)r_{s}$,
where $l_{s}:=L_{s}(f)(t)$ and $r_{s}:=R_{s}(f)(t)$ are constants
($\ne1_{G}$) for each $s\in S$. In particular, we have 
\[
f(0)^{-1}f(s+t)=f(0)^{-1}f(s)f(s)^{-1}f(t+s)=f(0)^{-1}f(s)f(0)^{-1}f(t),
\]
\[
f(s+t)f(0)^{-1}=f(s+t)f(t)^{-1}f(t)f(0)^{-1}=f(s)f(0)^{-1}f(t)f(0)^{-1}.
\]
Thus, $f(0)^{-1}f$ and $ff(0)^{-1}$ are both multiplicative functions.
\end{proof}
Just as a polynomial of degree $d$ is in general determined by $d+1$
polynomial values, certain special values of a polynomial map will
suffice to determine it. 
\begin{prop}
\label{prop:uniquely determined} Let $S_{0}$ be a set of generators
of a commutative monoid $S$ and $f:S\to G$ be a polynomial map of
degree $d$. Then, $f$ is uniquely determined by its values on $\{0\}\cup S_{0}^{\le d}$,
where \footnote{It is understood that the empty word (i.e., $i=0$) denotes the identity
element $0$. } 
\[
S_{0}^{\le d}=\{s_{1}+s_{2}+\ldots+s_{i}\mid1\le i\le d,s_{1},s_{2},\ldots,s_{d}\in S\}.
\]
If $S$ is a commutative semigroup without $0$ and is generated by
$S_{0}$, then $f$ is uniquely determined by its values on $S_{0}^{\le d+1}$. 

Furthermore, if $S$ is finitely generated commutative monoid or semigroup,
then the subgroup generated by the image of a polynomial map $f:S\to G$
is also finitely generated.
\end{prop}
\begin{proof}
The proof is by induction on the degree. If $d=-\infty,0$, then $f(s)$
is a constant for all $s\in S$. Suppose we have shown this for polynomial
maps of degree $<d$. Then, for any $s\in S_{0}$ and $t\in S_{0}^{\le d-1}$,
we have $s+t\in S_{0}^{\le d}$ and
\[
L_{s}(f)(t)=f(s+t)f(t)^{-1},\qquad R_{s}(f)(t)=f(t)^{-1}f(s+t)
\]
are polynomial maps of degree $\le d-1$. Since the values $f\restriction_{S_{0}^{\le d}}$
are given, the values $L_{s}(f)\restriction_{S_{0}^{\le d-1}}$ and
$R_{s}(f)\restriction_{S_{0}^{\le d-1}}$ are known. The induction
hypothesis implies that $L_{s}(f)$ and $R_{s}(f)$ are uniquely determined
for any $s\in S_{0}$. By the lemma below, they are also uniquely
determined for any $s\in S$. Hence, $f(s)=L_{s}(f)(0)f(0)=f(0)R_{t}(f)(0)$
is uniquely determined. 

A similar argument applies if $S$ is a commutative semigroup without
$0$, except in the last step. We write $S=\bigcup_{m=1}^{\infty}S_{0}^{\le m}$
and proceeds by induction on $m$. In this case, the values of $f$
on $S_{0}^{\le d+1}$ are given. If $s\in S_{0}^{\le m}$ with $m>d+1$,
then we write $s=u+v$, where $u\in S_{0}^{\le m-1}$ and $v\in S_{0}$.
Then, by induction, $f(s)=f(u+v)=L_{u}(f)(v)f(v)=f(v)R_{u}(f)(v)$
is uniquely determined. 

At last, the subgroup generated by the image of a polynomial map $f:S\to G$
is either $\langle f(S_{0}^{\le d})\cup\{f(0)\}\rangle$ or $\langle f(S_{0}^{\le d+1})\rangle$.
Since $S_{0}$ is finite, it is finitely generated.
\end{proof}
\begin{lem}
For any map $f:S\to G$ and any $s_{1},s_{2},t\in S$, we have 
\begin{align*}
L_{s_{1}+s_{2}}(f)(t) & =L_{s_{1}}(f)(s_{2}+t)L_{s_{2}}(f)(t)=L_{s_{2}}L_{s_{1}}(f)(t)L_{s_{1}}(f)(t)L_{s_{2}}(f)(t),\\
R_{s_{1}+s_{2}}(f)(t) & =R_{s_{1}}(f)(t)R_{s_{2}}(f)(s_{1}+t)=R_{s_{1}}(f)(t)R_{s_{2}}f(t)R_{s_{1}}R_{s_{2}}(f)(t).
\end{align*}
\end{lem}
\begin{proof}
By direct calculations. 
\end{proof}
 We can construct induced polynomial maps via homomorphism of either
semigroups or groups. 
\begin{prop}
\label{prop:homomorphism of commutative semigroups} Let $\phi:S_{0}\rightarrow S_{1}$
be a homomorphism of commutative semigroups and $f:S_{1}\rightarrow G$
be a polynomial map of  degree $d$. Then, the induced function $f^{*}=f\circ\phi:S_{0}\xrightarrow{\phi}S_{1}\xrightarrow{f}G$
is a polynomial map of  degree $\le d$. 

In particular, if $\phi:S_{1}\to S_{1}$ is an automorphism of commutative
semigroups, then the induced function $f^{*}=f\circ\phi$ has the
same degree as $f$. 
\end{prop}
\begin{proof}
By induction on the degree $d$.  In particular, if $\phi:S_{1}\to S_{1}$
is an automorphism, then the degree of $f=f^{*}\circ\phi^{-1}$ is
not larger than the degree of $f^{*}=f\circ\phi$. 
\end{proof}
\begin{prop}
\label{prop:homomorphism of groups}  Let $\phi:G\rightarrow H$
be a homomorphism (or antihomomorphism) of groups and $f:S\rightarrow G$
be a polynomial map of  degree $d$. Then, the induced function $f_{*}=\phi\circ f:S\xrightarrow{f}G\xrightarrow{\phi}H$
is a polynomial map of  degree $\le d$. 

In particular, if $\phi:G\rightarrow H$ is an isomorphism (or antiisomorphism)
of groups, then $f:S\rightarrow G$ is a polynomial map of  degree
$d$, if and only if $f_{*}$ is. 
\end{prop}
\begin{proof}
By induction on the degree $d$.  In particular, if $\phi:G\to H$
is an (anti)isomorphism, then the degree of $f=\phi^{-1}\circ f_{*}$
is not larger than the degree of $f_{*}=\phi\circ f$. 
\end{proof}
Let $G$ be any group.  Set $C^{1}G=G$ and inductively define $C^{i+1}G$
by $[C^{i}G,G]$ for all $i\ge1$. Then, the descending series 
\[
G=C^{1}G\ge C^{2}G\ge\cdots\ge C^{n}G\ge\cdots
\]
is the lower central series of $G$. Each $C^{n}G$ is normal  in
$G$ and $C^{n}G/C^{n+1}G$ is contained in the center of $G/C^{n+1}G$.
Moreover, the lower central series of a group is graded with respect
to commutators, i.e., for every $i,j\ge1$, we have $[C^{i}G,C^{j}G]\le C^{i+j}G$. 

Set $Z_{0}G=\{1_{G}\}$ and inductively define $Z_{i+1}G$ to be the
subgroup of $G$ such that $Z_{i+1}G/Z_{i}G=Z(G/Z_{i}G)$ for all
$i\ge1$. Then, the ascending series 
\[
\cdots\ge Z_{n}G\ge Z_{n-1}G\ge\cdots\ge Z_{1}G\ge Z_{0}G=\{1_{G}\},
\]
is the upper central series of $G$. Each $Z_{n}G$ is normal  in
$G$ and $Z_{1}G$ is the center of $G$. 

A group $G$ is said to be nilpotent if $G$ has a lower/upper central
series of finite length.  The smallest $n$ such that $G$ has a
lower/upper central series of length $n$ is called the nilpotency
class of $G$. 
\begin{defn}
A function $f:S\rightarrow G$ from a semigroup $S$ to a group $G$
is said to have lc-height $\ge k$ (relative to $G$), if the image
of $f$ lies in $C^{k}G$, and is said to have uc-height $\le k$
(relative to $G$), if the image of $f$ lies in $Z_{k}G$. \footnote{(Here, lc is short for lower central, and uc is short for upper central). }
\end{defn}
\begin{rem*}
 Since $C^{k}G\cdot C^{k'}G\subseteq C^{\min\{k,k'\}}G$, the product
$ff':S\rightarrow G$ of a function $f:S\to G$ of lc-height $\ge k$
and a function $f':S\to G$ of lc-height $\ge k'$ has lc-height $\ge\min\{k,k'\}$.
Since the lower central series of $G$ is graded with respect to commutators,
the commutator $[f,f']:S\to G$ is a function of lc-height $\ge k+k'$. 
\end{rem*}
Next, we state and prove a few corollaries of Proposition \ref{prop:homomorphism of groups}.
Since conjugation by a group element is an inner automorphism of the
group, we have 
\begin{cor}
\label{cor:f-conjugate} For any $f\in G$, the $f$-conjugate of
a polynomial map $f'$ of lc-height $k'$ and degree $d'$ 
\[
ff'f^{-1}:S\rightarrow G;\quad t\mapsto ff'(t)f^{-1}
\]
is a polynomial map of lc-height $k'$ and degree $d'$. 
\end{cor}
Notice that any translation $T_{s}(f)(t):=f(t+s)$ of a polynomial
map $f$ of lc-height $\ge k$ and degree $\le d$ by $s\in S$ is
a polynomial map of lc-height $\ge k$ and degree $\le d$. Similarly, 
\begin{cor}
\label{cor:f-translation} For any $f\in G\setminus\{1_{G}\}$ of
lc-height $\ge k$, the left $f$-translation $ff'$ (resp. the right
$f$-translation $f'f$) of a polynomial map $f'$ of lc-height $\ge k'$
and degree $\le d'$ is a polynomial map of lc-height $\ge\min\{k,k'\}$
and degree $\le\max\{0,d'\}$. 
\end{cor}
\begin{proof}
By induction on the degree $d'$, we prove this for $ff'$ and the
proof for $f'f$ is similar. The assertion clearly holds for $d'=-\infty,0$.
If $d'>0$, then 
\[
L_{s}(ff')(t)=(ff'(s+t))(ff'(t))^{-1}=ff'(s+t)f'(t)^{-1}f^{-1}=fL_{s}(f')(t)f^{-1}
\]
is the $f$-conjugate of a polynomial map $L_{s}(f')$ of degree $\le d'-1$,
and 
\[
R_{s}(ff')(t)=(ff'(t))^{-1}(ff'(s+t))=f'(t)^{-1}f^{-1}ff'(s+t)=R_{s}(f')(t)
\]
is a polynomial map of degree $\le d'-1$. By Corollary \ref{cor:f-conjugate},
$fL_{s}(f')f^{-1}$ is a polynomial map of degree $\le d'-1$. Hence,
$ff'$ is a polynomial map of degree $\le d'$. The assertion about
the lc-height is an easy consequence of the fact $C^{k}G\cdot C^{k'}G\subseteq C^{\min\{k,k'\}}G$. 
\end{proof}
\begin{rem*}
 The above corollary shows that the converse of Proposition \ref{prop:polynomial map of degree 1}
also holds. 
\end{rem*}
\begin{cor}
\label{cor:inverse} The composition $\iota\circ f:S\xrightarrow{f}G\xrightarrow{\iota}G$
of a polynomial map $f$ of lc-height $k$ and degree $d$ with the
inverse function $\iota:G\to G$ is a polynomial map of lc-height
$k$ and degree $d$. 
\end{cor}
\begin{cor}
\label{cor:quotient} Let $H$ be a normal subgroup of $G$ and $f:S\rightarrow G$
be a polynomial map of lc-height $\ge k$ and degree $\le d$. Then,
the induced function $f\mod H:S\xrightarrow{f}G\xrightarrow{\pi}G/H$
is also a polynomial map of lc-height $\ge k$ and degree $\le d$. 

If the induced polynomial map $f\mod H$ is of degree $\le0$, then
$f$ is a $g$-translation of a polynomial map $h:S\to H$ of degree
$\le d$, for some $g\in G$. (In particular, if $f\mod H$ has degree
$-\infty$, then we can take $f=h$ and $g=1_{G}$.) 
\end{cor}
\begin{proof}
The first assertion is a trivial consequence of Proposition \ref{prop:homomorphism of groups}.
If the induced polynomial map $f\mod H$ is of degree $0$, then the
image of $f$ lies in a left coset of $H$, say, $gH$ for some $g\in G\setminus H$.
Then, we can write $f=gh$, where $h:S\rightarrow H$ is some function;
If $f\mod H$ has degree $-\infty$, then we take $f=h$ and $g=1_{G}$.
If $d=-\infty$, then $f\mod H$ has degree $-\infty$, and if $d\ge0$,
then Corollary \ref{cor:f-translation} implies that $h=g^{-1}f:S\rightarrow H$
is a polynomial map of degree $\le\max\{0,d\}=d$ . Hence, $f$ is
a left $g$-translation of a polynomial map $h$ of degree $\le d$.
Similarly, $f$ can be written as a right $g$-translation of a polynomial
map of degree $\le d$, for some other $g\in G$. 
\end{proof}
\begin{rem*}
Suppose that $f:S\rightarrow G$ is a function and the induced function
$f_{*}=\phi\circ f:S\xrightarrow{f}G\xrightarrow{\phi}H$ is a polynomial
map of degree $\le d$. Then, what can we say about the function $f:S\rightarrow G$?
The answer is not much. Even when $d=-\infty,0$, i.e., $f_{*}$ is
a constant, we only know that $f:S\rightarrow G$ is a function such
that its image lies in some fiber $\phi^{-1}(h)$ for some $h\in H$.
\end{rem*}
 A natural question that one may ask is whether the (elementwise)
product 
\[
f_{1}f_{2}:S\rightarrow G;t\mapsto f_{1}(t)f_{2}(t)
\]
of two polynomial maps $f_{1}:S\to G$ and $f_{2}:S\to G$ a polynomial
map? The answer, in general, is no. The simplest example might be
provided by the multiplicative functions $f_{1}(n)=x^{n}$ and $f_{2}(n)=y^{n}$
from $\mathbb{N}$ to the free group $F_{2}$ generated by two generators
$x,y$. 
\begin{example}
\label{exa:Fibonacci}  Inspired by \cite[Example in Section 3.1]{Leibman2002},
we provide another beautiful example, which is  related to the Fibonacci
sequence $F_{n}$ with $F_{1}=F_{2}=1$. Consider the group 
\[
G=\langle x,y,z\mid[x,y]=1_{G},zxz^{-1}=yx,zyz^{-1}=x\rangle.
\]
Then, $f_{1}(n)=z^{n}x$ and $f_{2}(n)=z^{-n}$ are two polynomial
maps from $\mathbb{N}$ to $G$ of degree $1$. Let $f$ be the elementwise
product of $f_{1}$ and $f_{2}$, i.e., $f(n)=f_{1}(n)f_{2}(n)=z^{n}xz^{-n}$.
One can show by induction that $f(n)=x^{F_{n+1}}y^{F_{n}}$ for all
$n$. Therefore, the following identity 
\[
D_{1}(f(n))=f(n+1)f(n)^{-1}=x^{F_{n+2}}y^{F_{n+1}}x^{-F_{n+1}}y^{-F_{n}}=x^{F_{n}}y^{F_{n-1}}=f(n-1),
\]
implies that $f$ is not a polynomial map from $\mathbb{N}$ to either
$G$ or the normal abelian subgroup of $G$ generated by $x$ and
$y$. Roughly speaking, $f$ is more like an ``exponential map''.
\end{example}
\begin{rem*}
The group $G$ given in the above example has a normal abelian subgroup
$H$ generated by $x$ and $y$, such that $G/H$ is a cyclic group
generated by the image of $z$, and thus is metabelian, i.e., solvable
of derived length $2$. Example \ref{exa:Fibonacci} implies that
the product of two polynomial maps in metabelian (let alone solvable)
groups may not be a polynomial map. 
\end{rem*}
\begin{prop}
\label{prop:direct sum} Let $\Lambda$ be a finite index set. For
$\lambda\in\Lambda$, let $S_{\lambda}$ be a commutative semigroup
and $G_{\lambda}$ a nilpotent group of class $n_{\lambda}$. Then,
the direct sum $\bigoplus_{\lambda\in\Lambda}S_{\lambda}$ is a commutative
semigroup, the direct sum $\bigoplus_{\lambda\in\Lambda}G_{\lambda}$
is a nilpotent group of class $\max_{\lambda\in\Lambda}\{n_{\lambda}\}$,
and the direct sum 
\[
\bigoplus_{\lambda\in\Lambda}f_{\lambda}:\bigoplus_{\lambda\in\Lambda}S_{\lambda}\rightarrow\bigoplus_{\lambda\in\Lambda}G_{\lambda}
\]
of a finite family of polynomial maps $f_{\lambda}:S_{\lambda}\rightarrow G_{\lambda}$
of lc-height $\ge k_{\lambda}$ and degree $d_{\lambda}$ is a polynomial
map of lc-height $\ge\min_{\lambda\in\Lambda}\{k_{\lambda}\}$ and
degree $\max_{\lambda\in\Lambda}\{d_{\lambda}\}$. 
\end{prop}
\begin{proof}
The proof is trivial and thus omitted. 
\end{proof}
Given two polynomial maps $f:S\rightarrow G$ of degree $\le d$ and
$f':S\rightarrow G$ of degree $\le d'$, their product $ff':S\rightarrow G$
is given by the following composition 
\[
ff':S\xrightarrow{\Delta}S\times S\xrightarrow{f\times f'}G\times G\xrightarrow{m}G,
\]
where $\Delta:S\rightarrow S\times S$ is the diagonal map and $m:G\times G\rightarrow G$
is the multiplication map of $G$. Proposition \ref{prop:direct sum}
implies that $f\times f'$ is a polynomial of degree $\le\max\{d,d'\}$.
From Proposition \ref{prop:homomorphism of groups}, we see that if
$m:G\times G\rightarrow G$ or even $m:\langle f\times f'\rangle\to G$
is a homomorphism of groups, then the product $ff'$ will be a polynomial
map of degree $\le\max\{d,d'\}$. This simple idea allows us to prove
the following corollary. 
\begin{cor}
\label{cor:product lc-height k+k'>=00003Dn+1} Let $G$ be a nilpotent
group of class $n$. Given two polynomial maps $f:S\rightarrow G$
of degree $\le d$ and $f':S\rightarrow G$ of degree $\le d'$, if
$f$ has lc-height $\ge k$ and $f'$ has lc-height $\ge k'$ with
$k+k'\ge n+1$, then the product $ff'$ has lc-height $\ge\min\{k,k'\}$
and degree $\le\max\{d,d'\}$. 
\end{cor}
\begin{proof}
Indeed, for any $a,b,c,d\in G$, we have $m((a,b))m((c,d))=abcd$
and 
\begin{align*}
m((a,b)(c,d)) & =m((ac,bd))=acbd=abc[c^{-1},b^{-1}]d\\
 & =abcd[c^{-1},b^{-1}][[c^{-1},b^{-1}]^{-1},d^{-1}].
\end{align*}
 Since $[C^{k}G,C^{k'}G]=C^{n+1}G=\{1_{G}\}$, $m:C^{k}G\times C^{k'}G\to G$
is a homomorphism of groups. 
\end{proof}
 Now, we are ready to prove Theorem \ref{thm:product of two polynomial maps}. 
\begin{proof}[Proof of Theorem \ref{thm:product of two polynomial maps}]
 Replacing $G$ by the subgroup $\langle f,f'\rangle$, we may assume
that $G$ is nilpotent of class $n$. The proof is given by double
induction with the outer descending induction on $k+k'$, where $k,k'$
are lc-heights of the polynomial maps $f$ and $f'$ relative to $G=\langle f,f'\rangle$
respectively, and inner ordinary induction on $d_{f}+d_{f'}$, where
$d_{f}$ and $d_{f'}$ are respectively the induced degrees of the
polynomial maps $f\mod C^{k+1}G$ and $f'\mod C^{k'+1}G$. 

The idea is to apply the definition of polynomial maps and commutator
identities to create polynomial maps of either lower degrees or larger
lc-heights and move them to the correct place. Doing so creates extra
commutators of larger lc-heights, which is easy to deal with by the
induction hypothesis.

To facilitate the proof, the following statements will be proved simultaneously: 
\begin{enumerate}
\item \label{enu:product} The product $ff':S\rightarrow G$ is a polynomial
map; 
\item \label{enu:commutator} The commutator $[f,f']:S\rightarrow G$ is
a polynomial map. 
\end{enumerate}
Although it suffices to show (\ref{enu:product}), the proof is easier
if the induction hypothesis contains several claims, each of which
depends on the other for larger lc-heights or smaller induced degrees. 

Clearly, the theorem holds when $n\le1$, i.e., when $G$ is abelian.
Indeed, the commutator of two polynomial maps is always a polynomial
map of degree $-\infty$; the product $ff'$ is a polynomial map of
degree $\le\max\{d,d'\}$ by Corollary \ref{cor:product lc-height k+k'>=00003Dn+1}.
So we may assume that $n\ge2$. 

By Corollary \ref{cor:product lc-height k+k'>=00003Dn+1}, if $k+k'\ge n+1$,
then the product $ff'$ is a polynomial map of lc-height $\ge\min\{k,k'\}$
and degree $\le\max\{d,d'\}$, and the commutator $[f,f']$ is a polynomial
map of lc-height $n+1$ and degree $-\infty$. This gives the outer
induction base on the sum $k+k'$ of lc-heights $k$ and $k'$. 

Suppose that we have shown this for $k+k'>m$ with $2\le m\le n$.
The goal of the outer induction step is to prove the the claim holds
for $k+k'=m$. We proceed by the ordinary induction on  $(d_{f},d_{f'})$.
If either $d_{f}$ or $d_{f'}$ is $-\infty$, then we are in the
case when $k+k'=m+1$, which has been proved by induction hypothesis.
Therefore, we assume that $d_{f},d_{f'}\ge0$.

Then, for the product case, we have 
\begin{align}
(f(s+t)f'(s+t))(f(t)f'(t))^{-1} & =f(s+t)f'(s+t)f'(t)^{-1}f(t)^{-1}\label{eq:L_s(ff')(t)}\\
 & =f(s+t)L_{s}(f')(t)f(t)^{-1}\nonumber \\
 & =f(s+t)f(t)^{-1}L_{s}(f')(t)[L_{s}(f')(t)^{-1},f(t)]\nonumber \\
 & =L_{s}(f)(t)L_{s}(f')(t)[L_{s}(f')(t)^{-1},f(t)],\nonumber 
\end{align}
\begin{align}
(f(t)f'(t))^{-1}(f(s+t)f'(s+t)) & =f'(t)^{-1}f(t)^{-1}f(s+t)f'(s+t)\label{eq:R_s(ff')(t)}\\
 & =f'(t)^{-1}R_{s}(f)(t)f'(s+t)\nonumber \\
 & =[f'(t)^{-1},R_{s}(f)(t)]R_{s}(f)(t)f'(t)^{-1}f'(s+t)\nonumber \\
 & =[f'(t)^{-1},R_{s}(f)(t)]R_{s}(f)(t)R_{s}(f')(t).\nonumber 
\end{align}
By the inner induction hypothesis, $[L_{s}(f')^{-1},f]$ and $[f'^{-1},R_{s}(f)]$
are polynomial maps of lc-height $\ge k+k'=m$, and $L_{s}(f)L_{s}(f')$
and $R_{s}(f)R_{s}(f')$ are polynomial maps of lc-heights $\ge\min\{k,k'\}\ge1$,
 and thus $L_{s}(f)L_{s}(f')[L_{s}(f')^{-1},f]$ and $[f'^{-1},R_{s}(f)]R_{s}(f)R_{s}(f')$)
are polynomial maps of lc-height $\ge\min\{k,k'\}$. It follows that
$ff'$ is a polynomial map of lc-height $\ge\min\{k,k'\}$.

Next, we deal with the commutator case. We have 
\begin{align}
 & [f(s+t),f'(s+t)][f(t),f'(t)]^{-1}\label{eq:L_s(=00005Bf,f'=00005D)(t)}\\
= & f(s+t)f'(s+t)f(s+t)^{-1}f'(s+t)^{-1}f'(t)f(t)f'(t)^{-1}f(t)^{-1}\nonumber \\
= & f(s+t)f'(s+t)f(s+t)^{-1}R_{s}(f')(t)^{-1}f(t)f'(t)^{-1}f(t)^{-1}\nonumber \\
= & f(s+t)f'(s+t)f(s+t)^{-1}f(t)R_{s}(f')(t)^{-1}C_{1}(t)f'(t)^{-1}f(t)^{-1}\nonumber \\
= & f(s+t)f'(s+t)R_{s}(f)(t)^{-1}R_{s}(f')(t)^{-1}C_{1}(t)f'(t)^{-1}f(t)^{-1}\nonumber \\
= & f(s+t)f'(s+t)R_{s}(f)(t)^{-1}R_{s}(f')(t)^{-1}f'(t)^{-1}C_{2}(t)f(t)^{-1}\nonumber \\
= & f(s+t)f'(s+t)R_{s}(f)(t)^{-1}f'(s+t)^{-1}C_{2}(t)f(t)^{-1}\nonumber \\
= & f(s+t)f'(s+t)f'(s+t)^{-1}R_{s}(f)(t)^{-1}C_{3}(t)f(t)^{-1}\nonumber \\
= & f(t)C_{3}(t)f(t)^{-1}=C_{4}(t),\nonumber 
\end{align}
where by induction hypothesis
\begin{align*}
C_{1}(t) & =[R_{s}(f')(t),f(t)^{-1}], &  &  & C_{2}(t) & =C_{1}(t)[C_{1}(t)^{-1},f'(t)],\\
C_{3}(t) & =[R_{s}(f)(t),f'(s+t)]C_{2}(t), &  &  & C_{4}(t) & =[f(t),C_{3}(t)]C_{3}(t)
\end{align*}
are polynomial maps of lc-height $\ge k+k'$, and 
\begin{align}
 & [f(t),f'(t)]^{-1}[f(s+t),f'(s+t)]\label{eq:R_s(=00005Bf,f'=00005D)(t)}\\
= & f'(t)f(t)f'(t)^{-1}f(t)^{-1}f(s+t)f'(s+t)f(s+t)^{-1}f'(s+t)^{-1}\nonumber \\
= & f'(t)f(t)f'(t)^{-1}R_{s}(f)(t)f'(s+t)f(s+t)^{-1}f'(s+t)^{-1}\nonumber \\
= & f'(t)f(t)f'(t)^{-1}f'(s+t)R_{s}(f)(t)D_{1}(t)f(s+t)^{-1}f'(s+t)^{-1}\nonumber \\
= & f'(t)f(t)R_{s}(f')(t)R_{s}(f)(t)D_{1}(t)f(s+t)^{-1}f'(s+t)^{-1}\nonumber \\
= & f'(t)f(t)R_{s}(f')(t)R_{s}(f)(t)f(s+t)^{-1}D_{2}(t)f'(s+t)^{-1}\nonumber \\
= & f'(t)f(t)R_{s}(f')(t)f(t)^{-1}D_{2}(t)f'(s+t)^{-1}\nonumber \\
= & f'(t)f(t)f(t)^{-1}R_{s}(f')(t)D_{3}(t)f'(s+t)^{-1}\nonumber \\
= & f'(s+t)D_{3}(t)f'(s+t)^{-1}=D_{4}(t),\nonumber 
\end{align}
where by induction hypothesis
\begin{align*}
D_{1}(t) & =[R_{s}(f)(t)^{-1},f'(s+t)^{-1}], &  &  & D_{2}(t) & =D_{1}(t)[D_{1}(t)^{-1},f(s+t)],\\
D_{3}(t) & =[R_{s}(f')(t)^{-1},f(t)]D_{2}(t), &  &  & D_{4}(t) & =[f'(s+t),D_{3}(t)]D_{3}(t),
\end{align*}
are polynomial maps of lc-height $\ge k+k'$. So $[f,f']$ is a polynomial
map of lc-height $\ge k+k'$. 
\end{proof}
\begin{rem*}
Since each subgroup of a nilpotent group is nilpotent, Theorem \ref{thm:product of two polynomial maps}
holds when $G$ is nilpotent. 
\end{rem*}
\begin{cor}
\label{cor:group of polynomial maps} Let $S$ be a commutative semigroup
and $G$ a nilpotent group of class $n$. The set $G_{p}^{S}$ of
all polynomial maps $f:S\to G$ forms a nilpotent group of class $n$,
with group law given by elementwise multiplication.  In particular,
$G$ can be viewed as a subgroup of $G_{p}^{S}$. 
\end{cor}
\begin{proof}
It is an immediate consequence of Theorem \ref{thm:product of two polynomial maps}
that $G_{p}^{S}$ is a nilpotent group of class $\le n$. The subset
of all constant polynomial maps is easily seen to be a subgroup, which
is isomorphic to $G$. Thus, the nilpotency class of $G_{p}^{S}$
is at least $n$ and hence must be exactly $n$. 
\end{proof}
\begin{example*}
Up to isomorphism, there is only one semigroup $S$ with one element,
i.e., the singleton $\{s\}$ with operation $s\cdot s=s$. Then, any
polynomial $f:S\to G$ has degree $\le0$. If one identifies $f$
with its image $f(s)$, then the nilpotent group $G_{p}^{S}$ is seen
to be isomorphic to $G$.
\end{example*}
Recall that a group is $G$ is said to be locally nilpotent, if every
finitely generated subgroup of $G$ is nilpotent. Then, subgroups
and quotient groups of a locally nilpotent group are locally nilpotent
and the external product of two locally nilpotent groups is locally
nilpotent.  Since finitely generated subgroups of a nilpotent group
are nilpotent, nilpotent groups are locally nilpotent. Here is an
example of a locally nilpotent group which is not nilpotent. 
\begin{example*}
Let $p$ be a prime number. The Prüfer $p$-group $\mathbb{Q}_{p}/\mathbb{Z}_{p}$
can be viewed as the direct limit 
\[
\mathbb{Q}_{p}/\mathbb{Z}_{p}=\lim_{k\to\infty}\mathbb{Z}/p^{k}=\lim_{k\to\infty}\left(\mathbb{Z}/p\hookrightarrow\mathbb{Z}/p^{2}\hookrightarrow\mathbb{Z}/p^{3}\hookrightarrow\cdots\hookrightarrow\mathbb{Z}/p^{k}\hookrightarrow\cdots\right),
\]
where the embedding $\mathbb{Z}/p^{k}\hookrightarrow\mathbb{Z}/p^{k+1}$
is induced by multiplication by $p$. A presentation of $\mathbb{Q}_{p}/\mathbb{Z}_{p}$
is given by $\langle g_{1},g_{2},g_{3},\ldots\mid g_{1}^{p}=1,g_{2}^{p}=g_{1},g_{3}^{p}=g_{2},\dots\rangle$,
where the group operation is written as multiplication. Then, each
element of $\mathbb{Q}_{p}/\mathbb{Z}_{p}$ has $p$ different $p$th
roots. 

For any abelian group $H$, the generalized dihedral group corresponding
to $H$
\[
\Dih(H)=\langle H,s\mid s^{2}=(sh)^{2}=1,\forall h\in H\rangle,
\]
can be viewed as the semidirect product $H\rtimes_{\varphi}\mathbb{Z}/2$,
where $\varphi:\mathbb{Z}/2=\langle s\mid s^{2}=1\rangle\to\Aut(H)$
is a homomorphism\footnote{In the sequel, $\varphi$ will be understood and omitted.}
given by $\varphi(s)(h)=shs^{-1}=shs=h^{-1},\forall h\in H$.  The
dihedral groups $D_{2n}=\Dih(\mathbb{Z}/n)=\langle r,s\mid r^{n}=s^{2}=(sr)^{2}=1\rangle$
are special cases of generalized dihedral groups and $D_{2n}$ is
nilpotent if and only if it has order $2n=2^{k}$ for some $k\in\mathbb{N}$,
and $D_{2\cdot1}=\mathbb{Z}/2$ and $D_{2\cdot2}=\mathbb{Z}/2\times\mathbb{Z}/2$
are the only abelian ones.  Moreover, $D_{2\cdot2^{k}}$ is of nilpotency
class $k$ for $k\ge1$. 

Then, the generalized dihedral group $\Dih(\mathbb{Q}_{2}/\mathbb{Z}_{2})=\mathbb{Q}_{2}/\mathbb{Z}_{2}\rtimes\mathbb{Z}/2$
corresponding to the Prüfer $2$-group is an example of a locally
nilpotent group which is not nilpotent. Then, one sees that 
\[
\lim_{k\to\infty}D_{2\cdot2^{k}}=\lim_{k\to\infty}(\mathbb{Z}/2^{k}\rtimes\mathbb{Z}/2)\le(\lim_{k\to\infty}\mathbb{Z}/2^{k})\rtimes\mathbb{Z}/2=\Dih(\mathbb{Q}_{2}/\mathbb{Z}_{2}).
\]
It is locally nilpotent as every finitely generated subgroup must
be contained in some finite nilpotent subgroup $D_{2\cdot2^{k}}=\mathbb{Z}/2^{k}\rtimes\mathbb{Z}/2$
and thus is nilpotent. Since one can find nilpotent subgroups $D_{2\cdot2^{k}}$
of nilpotency class $k$ for arbitrarily large $k$, $\Dih(\mathbb{Q}_{2}/\mathbb{Z}_{2})$
cannot be nilpotent. 
\end{example*}
A trivial consequence of Theorem \ref{thm:product of two polynomial maps}
is that the product of two polynomial maps $f_{1},f_{2}:S\to G$ in
a locally nilpotent group $G$ is a polynomial map, if the subgroup
$\langle f_{1},f_{2}\rangle$ generated by $f_{1}(S)$ and $f_{2}(S)$
is finitely generated. Moreover, we have 
\begin{prop}
\label{prop:ploynomial map in locally nilpotent subgroup} If $g:S\to G$
is a polynomial map from a finitely generated commutative semigroup
$S$ to a locally nilpotent group $G$, then the subgroup $\langle g\rangle$
is finitely generated and nilpotent.
\end{prop}
\begin{proof}
This follows easily from the Proposition \ref{prop:uniquely determined}
and the definition of a locally nilpotent group.
\end{proof}
But if the commutative semigroup is not finitely generated, we can
construct an example such that the product of two polynomial maps
in locally nilpotent groups may not be a polynomial map. The motivation
is that the alternating sequence $f:\mathbb{N}\to\mathbb{Z}$ sending
$n$ to $(-1)^{n}$ cannot be a polynomial map. Indeed, $D_{1}f(n)=2\cdot(-1)^{n-1}$,
$D_{1}D_{1}f(n)=2^{2}\cdot(-1)^{n}$, etc. 
\begin{example}
\label{exa:polynomials maps in locally nilpotent groups}  Let $F_{\mathbb{N}}$
be the free abelian group on the generators $x_{1},x_{2},\ldots$.
Then, each element $x=\sum_{i=1}^{\infty}n_{i}x_{i}\in F_{\mathbb{N}}$
can be uniquely written as an infinite dimensional vector $(n_{1},n_{2},\ldots)$
with only finitely many nonzero $n_{i}$. Let $\varepsilon:F_{\mathbb{N}}\to\mathbb{Z}$
be the augmentation given by $\varepsilon(\sum_{i=1}^{\infty}n_{i}x_{i})=\sum_{i=1}^{\infty}n_{i}$. 

Let $\langle g_{1},g_{2},g_{3},\ldots\mid g_{1}^{2}=1,g_{2}^{2}=g_{1},g_{3}^{2}=g_{2},\dots\rangle$
be a presentation of the Prüfer $2$-group $\mathbb{Q}_{2}/\mathbb{Z}_{2}$.
Consider the homomorphism $\varphi_{0}:F_{\mathbb{N}}\to\Dih(\mathbb{Q}_{2}/\mathbb{Z}_{2})=\mathbb{Q}_{2}/\mathbb{Z}_{2}\rtimes\mathbb{Z}/2$
such that $\varphi_{0}(x_{i})=(g_{i},0)$ and the following composition
of homomorphisms 
\[
\varphi_{1}:F_{\mathbb{N}}\xrightarrow{\varepsilon}\mathbb{Z}\twoheadrightarrow\mathbb{Z}/2\hookrightarrow\mathbb{Q}_{2}/\mathbb{Z}_{2}\rtimes\mathbb{Z}/2,
\]
where the last arrow embeds $\mathbb{Z}/2$ as the second factor of
$\mathbb{Q}_{2}/\mathbb{Z}_{2}\rtimes\mathbb{Z}/2$, i.e., $\varphi_{1}(x)=(0,\varepsilon(x)\mod2)$.
Then, $\varphi_{0}$ and $\varphi_{1}$ are polynomial maps of degree
$1$. The subgroup $\langle\varphi_{0}\rangle$ is isomorphic to $\mathbb{Q}_{2}/\mathbb{Z}_{2}$,
which is not finitely generated, and the subgroup $\langle\varphi_{1}\rangle$
is $\{0\}\rtimes\mathbb{Z}/2$. 

Then, the product $\varphi=\varphi_{0}\varphi_{1}:F_{\mathbb{N}}\to\Dih(\mathbb{Q}_{2}/\mathbb{Z}_{2})$
is not a polynomial map. Indeed, for any $s=\sum_{i=1}^{\infty}s_{i}x_{i}\in F_{\mathbb{N}}$,
we have $\varphi(s)=\varphi_{0}(s)\varphi_{1}(s)=\left(\prod g_{i}^{s_{i}},\varepsilon(s)\mod2\right)$.
Then, for any $t=\sum_{i=1}^{\infty}t_{i}x_{i}$, we have the following:
\[
L_{s}(\varphi)(t)=\varphi(s+t)\varphi(t)^{-1}=\begin{cases}
\left(\prod g_{i}^{s_{i}},\varepsilon(s)\mod2\right), & \text{if }\varepsilon(s)\equiv0\mod2,\\
\left(\prod g_{i}^{s_{i}+2t_{i}},\varepsilon(s)\mod2\right), & \text{if }\varepsilon(s)\equiv1\mod2,
\end{cases}
\]
\[
R_{s}(\varphi)(t)=\varphi(t)^{-1}\varphi(s+t)=\begin{cases}
\left(\prod g_{i}^{s_{i}},\varepsilon(s)\mod2\right), & \text{if }\varepsilon(t)\equiv0\mod2,\\
\left(\prod g_{i}^{-s_{i}},\varepsilon(s)\mod2\right), & \text{if }\varepsilon(t)\equiv1\mod2.
\end{cases}
\]
In particular, fixing $i\in\mathbb{N}$ and $s\in F_{\mathbb{N}}$
such that $\varepsilon(s)=0\mod2$, choosing $t$ among the sequence
$x_{i},2x_{i},3x_{i},\ldots,nx_{i},\ldots$, we have 
\begin{align*}
R_{s}(\varphi)(nx_{i}) & =\begin{cases}
\left(\prod g_{i}^{s_{i}},0\right), & \text{if }n\equiv0\mod2,\\
\left(\prod g_{i}^{-s_{i}},0\right), & \text{if }n\equiv1\mod2,
\end{cases}\\
D_{x_{i}}R_{s}(\varphi)(nx_{i}) & =\begin{cases}
\left(\prod g_{i}^{-2s_{i}},0\right), & \text{if }n\equiv0\mod2,\\
\left(\prod g_{i}^{2s_{i}},0\right), & \text{if }n\equiv1\mod2,
\end{cases}
\end{align*}
where $D$ is either $L$ or $R$. Then, repeating this for $2m$
times, we see that 
\begin{align*}
\underbrace{D_{x_{i}}\cdots D_{x_{i}}}_{2m}R_{s}(\varphi)(nx_{i}) & =\begin{cases}
\left(\prod g_{i}^{2^{2m}s_{i}},0\right), & \text{if }n\equiv0\mod2,\\
\left(\prod g_{i}^{-2^{2m}s_{i}},0\right), & \text{if }n\equiv1\mod2,
\end{cases}\\
 & =(1,0),\quad\text{for all sufficiently large }m,
\end{align*}
because $\varphi_{0}(s)=\prod g_{i}^{s_{i}}$ and its inverse have
finite order. However, $\varphi_{0}(s)=(\prod g_{i}^{s_{i}},0)$ may
have arbitrarily large finite order. This implies that $\varphi(t)$
is not a polynomial map. 
\end{example}
\begin{rem}
\label{rem:left but not right polynomial} Notice that $L_{s}(\varphi)$
behaves in a completely different way when we apply difference operators
to it. In fact, $L_{s}(\varphi)$ is a polynomial of degree $\le1$
for any $s\in F_{\mathbb{N}}$, and thus $\varphi$ is a left polynomial
of degree $\le2$. If $\varepsilon(s)\equiv0\mod2$, we have $L_{s}(\varphi)(t)=\left(\prod g_{i}^{s_{i}},0\right)$
and thus $D_{u}L_{s}(\varphi(t))=(1,0)$ for any $u\in F_{\mathbb{N}}$.
If $\varepsilon(s)\equiv1\mod2$, we have $L_{s}(\varphi)(t)=\left(\prod g_{i}^{s_{i}+2t_{i}},1\right)$
and  
\[
L_{u}L_{s}(\varphi(t))=\left(\prod g_{i}^{2u_{i}},0\right),\qquad R_{u}L_{s}(\varphi(t))=\left(\prod g_{i}^{-2u_{i}},0\right).
\]
Thus, for any $v\in F_{\mathbb{N}}$, we have $D_{v}L_{u}L_{s}(\varphi(t))=(1,0)$
and $D_{v}R_{u}L_{s}(\varphi(t))=(1,0)$. 

Following Leibman's definition, we may also define the left-polynomial
and right-polynomial from a commutative semigroup $S$ to any group
$G$ in the same manner. These definitions coincide if $S$ is an
abelian group. Then, the above example shows that $\varphi$ is a
left-polynomial of degree $\le2$, but not a right-polynomial. 
\end{rem}
Notice that in Theorem \ref{thm:product of two polynomial maps},
we have no estimate of the degree of the product of two general polynomial
maps. But with the following concept of lc-degree, we may strengthen
Corollary \ref{cor:group of polynomial maps}.
\begin{defn}
Let $G$ be a nilpotent group of class $n$ and $f:S\to G$ be a polynomial
map of degree $d$. By Corollary \ref{cor:quotient}, $f\mod C^{i+1}G$
is a polynomial map of degree $\le d$ for each $i=1,2,\ldots,n$.
Let $d_{i}$ be the degree of $f\mod C^{i+1}G$, i.e., $d_{i}$ is
least number in $\mathbb{Z}_{*}$ such that for any $s_{1},s_{2},\ldots,s_{d_{i}+1}\in S$,
\[
D_{s_{1}}D_{s_{2}}\cdots D_{s_{d_{i}+1}}f(S)\subseteq C^{i+1}G.
\]
Then, $\hat{d}=(d_{1},\ldots,d_{n})\in\mathbb{Z}_{*}^{n}$ will be
called the lc-degree of $f$.
\end{defn}
Leibman \cite{Leibman2002} has a slightly different definition for
lc-degree using the superadditive vectors: 
\begin{defn*}
A vector $\bar{d}=(d_{1},\ldots,d_{n})\in\mathbb{Z}_{*}^{n}$ is said
to be superadditive, if $d_{i}\le d_{j}$ for all $i\le j$, and $d_{i}+d_{j}\le d_{i+j}$,
for all $i,j$ with $i+j\le n$. 
\end{defn*}
With the obvious lexicographical order on $\mathbb{Z}_{*}^{n}$, there
exists a unique smallest superadditive vector $\bar{d}\ge\hat{d}=(d_{1},\ldots,d_{n})$
and Leibman calls $\bar{d}$ the lc-degree of $f$. To distinguish
them, our lc-degree $\hat{d}$ does not have to be superadditive.

Then, for any $c\in\mathbb{N}_{0}$, we see that $(a-b)-c=a-(b+c)$.
For any superaddtive vector $\bar{d}=(d_{1},\ldots,d_{n})\in\mathbb{Z}_{*}^{n}$
and any $b\in\mathbb{N}_{0}$, we define $\bar{d}-b=(d_{1}-b,\ldots,d_{n}-b)$.
Then, we see that if $\bar{d}$ is superaddtive, then $\bar{d}-b$
is also superaddtive, and $(\bar{d}-b)-c=\bar{d}-(b+c)$. However,
if $\bar{d}$ is superadditive, $\bar{d}+a=(d_{1}+a,\ldots,d_{n}+a)$
may not be superadditive. 

Now let us focus on our definition of lc-degree. Notice that $d_{i}\le d_{j}$
for all $i\le j$ and that $d_{i}=-\infty$, if and only if $f(S)\subseteq C^{i+1}G$,
if and only if the lc-height of $f$ is $\ge i+1$. If the lc-degree
of $f$ is $\hat{d}=(d_{1},\ldots,d_{n})$, then the degree of $f$
is $d_{n}$; conversely, if the degree of $f$ is $d$, then $\hat{d}\le(d,d,\ldots,d)$.
Also, if $f:S\to G$ is a polynomial map of lc-degree $\le\hat{d}$,
then $L_{s}(f)$ and $R_{s}(f)$ are polynomial maps of lc-degree
$\le\hat{d}-1$ for any $s\in S$. By Corollary \ref{cor:inverse},
if $f\mod C^{i+1}G$ has degree $\le d_{i}$, then $f^{-1}\mod C^{i+1}G$
has degree $\le d_{i}$ and vice versa. So $f$ and $f^{-1}$ have
the same lc-degree. 

Moreover, for all $s\in S$ if both $L_{s}(f)\mod C^{i+1}G$ and $R_{s}(f)\mod C^{i+1}G$
have degree $\le d_{i}$, then what can we say about the degree of
$f\mod C^{i+1}G$? Well, if $d_{i}\ge0$, then the degree of $f\mod C^{i+1}G$
must be $\le d_{i}+1$ and equality holds if at least one of $L_{s}(f)\mod C^{i+1}G$
and $R_{s}(f)\mod C^{i+1}G$ has degree $d_{i}$ for some $s\in S$;
if $d_{i}=-\infty$, which means that $L_{s}(f)(t)$ and $R_{s}(f)(t)$
all lie in $C^{i+1}G$ and thus $f(s+t)\equiv f(t)\mod C^{i+1}G$
for all $s,t\in S$, then the degree of $f\mod C^{i+1}G$ must be
$\le0$. So if $L_{s}(f)$ and $R_{s}(f)$ are polynomial maps of
lc-degree $\le\hat{d'}=(d_{1}',\ldots,d_{n}')$ for all $s\in S$,
then $f$ is a polynomial map of lc-degree $\le\hat{d}=(d_{1},\ldots,d_{n})$,
where 
\[
d_{i}=\begin{cases}
d_{i}'+1, & \text{if }d_{i}'\ge0,\\
0 & \text{if }d_{i}'=-\infty.
\end{cases}
\]
So even if one starts with a superadditive vector $\hat{d'}$, there
is no reason to expect $\hat{d}$ to be superadditive, because it
may happen that $d_{i}+d_{j}=d_{i}'+1+d_{j}'+1\not\leq d_{i+j}'+1=d_{i+j}$.

Furthermore, if $f$ has lc-degree $\hat{d}=(d_{1},\ldots,d_{n})$
and $f'$ has lc-degree $\hat{d'}=(d_{1}',\ldots,d_{n}')$, then what
can we say about $[f,f']$? Let $i$ (resp. $j$) be the largest index
such that $d_{i}\ge0$ (resp. $d_{j}'\ge0$). Then, we see that $f$
has lc-height $i$ and $f'$ has lc-height $j$. Hence, $[f,f']$
has lc-height $i+j$.

Now we are ready to strengthen Corollary \ref{cor:group of polynomial maps},
whose proof is similar to the one of Theorem \ref{thm:product of two polynomial maps}
and to the one of \cite[Proposition 3.4.]{Leibman2002}.  
\begin{cor}
\label{cor:lc-degree <=00003D a superadditive vector} Let $S$ be
a commutative semigroup and $G$ a nilpotent group of class $n$.
The set of all polynomial maps $f:S\to G$ with lc-degree $\hat{d}$
bounded above by some fixed superadditive vector $\bar{d}=(d_{1},\ldots,d_{n})$,
i.e., $\hat{d}\le\bar{d}$, forms a nilpotent subgroup of $G_{p}^{S}$.
\end{cor}
\begin{proof}
Let $b,c,c'\in\mathbb{N}_{0}$. We will prove the following statements
by descending induction on $b$: 
\begin{enumerate}
\item \label{enu:product-lc-degree} If $f,f'$ are polynomial maps of lc-degree
$\le\bar{d}-b$, then $ff'$ is a polynomial map of lc-degree $\le\bar{d}-b$; 
\item \label{enu:commutator-lc-degree} If $f,f'$ are polynomial maps of
lc-degree $\hat{d}\le\bar{d}-c$ and $\hat{d'}\le\bar{d}-c'$, then
$[f,f']$ is a polynomial map of lc-degree $\le\bar{d}-(c+c')$. 
\item \label{enu:inverse-lc-degree} If $f$ is a polynomial map of lc-degree
$\hat{d}\le\bar{d}-b$, then $f^{-1}$ is a polynomial of lc-degree
$\hat{d}\le\bar{d}-b$. 
\end{enumerate}
Notice that the last statement has been proved, since $f$ and $f^{-1}$
have the same lc-degree. Also, if $b$ is large enough ($b>d_{n})$,
then a polynomial map has lc-degree $\le\bar{d}-b=(-\infty,\ldots,-\infty)$
implies that it has degree $-\infty$. Then, assertion (\ref{enu:product-lc-degree})
is trivially satisfied, while by the same reason assertion (\ref{enu:commutator-lc-degree})
is trivially satisfied if $b+b'>2d_{n}$, as in this case either $b>d_{n}$
or $b'>d_{n}$. So this gives the induction base. 

Suppose that we have shown that assertions (\ref{enu:product-lc-degree})
and (\ref{enu:commutator-lc-degree}) hold for $b>m$ and $c+c'>m$
for some $m\in\mathbb{N}_{0}$. The goal of the induction step is
to show that they also hold for $b=m$ and $c+c'=m$. 

Then, for the product case, we take a closer look at equations (\ref{eq:L_s(ff')(t)})
and (\ref{eq:R_s(ff')(t)}). By the induction hypothesis, $L_{s}(f)$,
$L_{s}(f')$, $R_{s}(f)$ and $R_{s}(f')$ are polynomial maps of
lc-degree $\le\bar{d}-(m+1)$, and $[L_{s}(f')^{-1},f]$ and $[f'^{-1},R_{s}(f)]$
are polynomial maps of lc-degree $\le\bar{d}-(2m+1)\le\bar{d}-(m+1)$.
So $ff'$ is a polynomial map of lc-degree $\le(a_{1},\ldots,a_{n})$,
where $a_{i}=d_{i}-m$ if $d_{i}\ge m$. If there exists $i$ such
that $d_{i}<m$, then the assumption that $f\mod C^{i+1}G$ and $f'\mod C^{i+1}G$
have degree $\le d_{i}-m=-\infty$ implies that $f(t)\in C^{i+1}G$
and $f'(t)\in C^{i+1}G$, thus $f(t)f'(t)\in C^{i+1}G$, i.e., $ff'\mod C^{i+1}G$
has degree $-\infty$, so $a_{i}=-\infty=d_{i}-m$. Hence, it follows
that $ff'$ is a polynomial map of lc-degree $\le\bar{d}-m$. 

For the commutator case, we check equations (\ref{eq:L_s(=00005Bf,f'=00005D)(t)})
and (\ref{eq:R_s(=00005Bf,f'=00005D)(t)}). With the induction hypothesis,
it is easy to see that $C_{1},C_{2},C_{3},D_{1},D_{2},D_{3}$, $L_{s}([f,f'])=C_{4}$
and $R_{s}([f,f'])=D_{4}$ are polynomial maps of lc-degree $\le\bar{d}-(c+c'+1)=\bar{d}-(m+1)$.
Therefore, $[f,f']$ is a polynomial map of lc-degree $\le(a_{1},\ldots,a_{n})$,
where $a_{i}=d_{i}-m$ if $d_{i}\ge m$.  If there exists $i$ such
that $d_{i}<m$, we let $j$ (resp. $j'$) be the smallest index such
that $f\mod C^{j+1}G$ (resp. $f'\mod C^{j'+1}G$) has nonnegative
degrees. Thus, this implies that $f(t)\in C^{j}G$ and $f'(t)\in C^{j'}G$
for all $t\in S$, and that $d_{j}-c\ge0$ and $d_{j'}-c'\ge0$. If
$j+j'\ge n+1$, then $[f(t),f'(t)]\in C^{j+j'}G=\{1_{G}\}$, so there
is nothing to prove. If $j+j'\le n$, then we have \footnote{This is where we need the superadditivity.}
\[
d_{j+j'}\ge d_{j}+d_{j'}\ge c+c'=m>d_{i},
\]
and thus $j+j'>i$. So it follows that $[f(t),f'(t)]\in C^{j+j'}G\subseteq C^{i+1}G$,
i.e., $a_{i}=-\infty=d_{i}-m$. Hence, it follows that $[f,f']$ is
a polynomial map of lc-degree $\le\bar{d}-m$. 
\end{proof}
Let $S$ and $S'$ be commutative semigroups. Then, any polynomial
map $f:S\to G$ can be viewed as a polynomial map $\tilde{f}:S\times S'\to G$
by sending $(s,s')$ to $f(s)$. One checks that $\tilde{f}$ is a
polynomial map and shares many of the same quantities with $f$, such
as degree, lc-height, lc-degree, etc. 
\begin{defn}
We say that $\tilde{f}:S\times S'\to G$ defined above is a polynomial
map lifted from $f:S\to G$ to the commutative semigroup $S\times S'$.
If $\tilde{f}':S\times S'\to G$ is lifted from another polynomial
map $f':S'\to G$, then we can define an ordered product of $f$ and
$f'$ by the following formula:
\[
f\odot f':S\times S'\to G;\quad(s,s')\mapsto\tilde{f}(s,s')\tilde{f}'(s,s')=f(s)f'(s').
\]
\end{defn}
The ordered product of two polynomial maps will be a polynomial map
in some good cases, for example when $G$ is nilpotent. But the most
interesting result will be the following
\begin{cor}
\label{cor:ordered product of a polynomial map} Let $S$ be a commutative
semigroup and $G$ a nilpotent group of class $n$. Let $f:S\to G$
be any polynomial map, whose lc-degree is bounded by a fixed superadditive
vector $\bar{d}$. Then, for any $k\in\mathbb{N}$, the ordered product
$\bigodot_{i=1}^{k}f:\bigoplus_{i=1}^{k}S\to G$ is a polynomial map
whose lc-degree is bounded by the same superadditive vector $\bar{d}$. 

Moreover, if the group $\langle f\rangle$ generated by $f(S)$ is
finitely generated, then the subgroup $\langle\bigodot_{i=1}^{k}f\rangle$
generated by $\bigodot_{i=1}^{k}f(\bigoplus_{i=1}^{k}S)$ is of finite
index in $\langle f\rangle$.
\end{cor}
\begin{proof}
By induction on $k$, the first assertion follows from Corollary \ref{cor:lc-degree <=00003D a superadditive vector}.
 The subgroup $\langle\bigodot_{i=1}^{k}f\rangle$ contains the subgroup
generated by the image of $\bigodot_{i=1}^{k}f$ on the diagonal $S$
of $\bigoplus_{i=1}^{k}S$, which is nothing but $\langle f^{k}\rangle$.
Since $\langle f\rangle$ is finitely generated and nilpotent, by
a result due to Mal'tsev (cf. \cite[Thm 2.23]{CMZ2017}), $\langle f^{k}\rangle$
has finite index in $\langle f\rangle$ and so is $\langle\bigodot_{i=1}^{k}f\rangle$.
\end{proof}

\section{Continuous Polynomial Maps \label{sec:Continuous Polynomial Maps}}

Any usual polynomial $f:\mathbb{R}_{\ge0}\rightarrow\mathbb{R}$ of
degree $\le d$ \footnote{The degree of the zero polynomial is either left undefined, or defined
to be negative (usually $-1$ or $-\infty$). But to be compatible
with our definition of the degree of polynomial maps, it will be defined
to be $-\infty$ as well. } is a polynomial map of degree $\le d$ in our sense. The converse
also holds, provided that the polynomial map $f:\mathbb{R}_{\ge0}\rightarrow\mathbb{R}$
is assumed to be continuous. But before proving this, we need a few
definitions and lemmas. 
\begin{defn}
Let $S$ be a commutative semigroup and $G$ be an abelian group.
A function $f:S\rightarrow G$ is called additive, if it satisfies
Cauchy's functional equation: 
\begin{equation}
f(s+t)=f(s)+f(t),\qquad\forall s,t\in S.\label{eq:Cauchy}
\end{equation}
\end{defn}
\begin{rem*}
If $\{f_{i}:S\to G\mid1\le i\le n\}$ is a finite family of additive
functions, then so is their linear combination $\sum_{i=1}^{n}r_{i}f_{i}$,
where $r_{i}\in\mathbb{Z}$. 
\end{rem*}
The following property is fundamental for additive functions $f:S\to G$: 
\begin{lem}
\label{lem:additive =00003D> linear} Let $f:S\to G$ be an additive
function. 
\begin{enumerate}
\item The function $f$ is always $\mathbb{N}$-linear, i.e., for all $n\in\mathbb{N}$
and all $s_{1},\ldots,s_{n}\in S$, we have $f\left(\sum_{i=1}^{n}s_{i}\right)=\sum_{i=1}^{n}f(s_{i})$
and \textup{$f(ns)=nf(s)$, $\forall n\in\mathbb{N}$ and all $s\in S$. }
\item If $S$ is a commutative monoid, then $f$ is $\mathbb{N}_{0}$-linear
and in particular $f(0)=0$. 
\item If $S$ is an abelian group, then $f$ is $\mathbb{Z}$-linear. 
\end{enumerate}
For the following three assertions, we need to assume that $G$ is
torsion-free. 
\begin{enumerate}
\item[$(1')$]  If $S$ is a uniquely divisible commutative semigroup, i.e., for
each $s\in S$ and each $n\in\mathbb{N}$, there exists a unique $t\in S$,
such that $s=nt$, then $f$ is $\mathbb{Q}_{>0}$-linear. 
\item[$(2')$]  If $S$ is a uniquely divisible commutative monoid, then $f$ is
$\mathbb{Q}_{\ge0}$-linear. 
\item[$(3')$]  If $S$ is a divisible abelian group, then $f$ is $\mathbb{Q}$-linear. 
\end{enumerate}
\end{lem}
\begin{proof}
The proofs are all very similar, cf. \cite[Thm 5.2.1]{Kuczma2009Functional}.
Here we only show the last one. By induction, for all $n\in\mathbb{N}$
and all $s_{1},\ldots,s_{n}\in S$, we have $f\left(\sum_{i=1}^{n}s_{i}\right)=\sum_{i=1}^{n}f(s_{i})$.
Letting $s_{1}=s_{2}=\cdots=s_{n}=s$, we see that $f(ns)=nf(s)$
for all $n\in\mathbb{N}$, and thus $f$ is $\mathbb{N}$-linear.
Since $f(0)=f(0+0)=f(0)+f(0)$, we see that $f(0)=0$, and thus $f$
is $\mathbb{N}_{0}$-linear. Since 
\[
0=f(0)=f(s-s)=f(s)+f(-s),
\]
we see that $f(s)=-f(s)$, i.e., $f$ is an odd function, and thus
is $\mathbb{Z}$-linear. Since any $\lambda\in\mathbb{Q}$ can be
written as $\lambda=n/m$, where $n\in\mathbb{Z}$ and $m\in\mathbb{N}$,
and $S$ is divisible, for all $s\in S$, there exists a unique $t\in S$
such that $mt=ns$, we write $t=ns/m=\lambda s$ and thus have 
\[
mf(\lambda s)=mf(t)=f(mt)=f(ns)=nf(s),
\]
for all $s\in S$. Since $G$ is torsion-free, we have $f(\lambda s)=\lambda f(s)$
for all $s\in S$, and thus $f$ is $\mathbb{Q}$-linear. 
\end{proof}
\begin{rem*}
Uniqueness is needed here to make sense of $\lambda s$ for $\lambda\in\mathbb{Q}$. 
\end{rem*}
\begin{thm}
\label{thm:polynomial in 1 var} Every continuous polynomial map $f:\mathbb{R}_{\ge0}\rightarrow\mathbb{R}$
of is the usual polynomial. 
\end{thm}
\begin{proof}
By the continuity of $f$, it suffices to prove that the restriction
$f:\mathbb{Q}_{\ge0}\to\mathbb{R}$ is a polynomial. The proof is
by induction on the degree $d$. It is clear for $d\le0$. If $f$
is a polynomial map of degree $\le1$, then for all $s\in\mathbb{Q}_{\ge0}$,
$f(s+t)-f(t)$ is a polynomial map of degree $\le0$, i.e., a constant,
say $C_{s}$. For any $s_{1},s_{2}\in\mathbb{Q}_{\ge0}$, we have
\[
C_{s_{1}+s_{2}}=f(s_{1}+s_{2}+t)-f(t)=f(s_{1}+s_{2}+t)-f(s_{2}+t)+f(s_{2}+t)-f(t)=C_{s_{1}}+C_{s_{2}}.
\]
This implies that $C_{s}$ is additive with respect to $s$. By Lemma
\ref{lem:additive =00003D> linear}, $C_{\lambda s}=\lambda C_{s}$
for all $\lambda\in\mathbb{Q}_{\ge0}$. Setting $k=C_{1}$, we obtain
$C_{s}=sC_{1}=ks$ for all $s\in\mathbb{Q}_{\ge0}$.  Then, $f(s)=C_{s}+f(0)=ks+f(0)$
for all $s\in\mathbb{Q}_{\ge0}$. Hence, $f(t)=kt+f(0)$ is a polynomial
of degree $\le1$. 

Suppose that the assertion holds when $d\le n-1$ with $n\ge2$. Let
$f$ be a polynomial map of degree $\le n$.  By Corollary \ref{cor:product lc-height k+k'>=00003Dn+1},
for any $A_{n}\in\mathbb{R}$, $f(t)-A_{n}t^{n}$ is a polynomial
map of degree $\le n$. We claim that there exists some $A_{n}\in\mathbb{R}$,
such that $f(t)-A_{n}t^{n}$ is a polynomial map of degree $\le n-1$.
By the induction hypothesis, $f(t)-A_{n}t^{n}$ is a polynomial of
degree $\le n-1$, and therefore $f(t)$ is a polynomial of degree
$\le n$. So it suffices to find such an $A_{n}$. By definition,
for all $s\in\mathbb{Q}_{\ge0}$, $P_{s}(t):=f(s+t)-f(t)$ is a polynomial
map of degree $\le n-1$. Then, by the induction hypothesis, 
\[
P_{s}(t)=\alpha_{0}(s)+\sum_{i=1}^{n-1}\alpha_{i}(s)t^{i}
\]
is a polynomial in variable $t$ of degree $\le n-1$, where $\alpha_{i}(s)$
are functions from $\mathbb{Q}_{\ge0}$ to $\mathbb{R}$. Then, it
is a routine check that $A_{n}:=\dfrac{1}{n}\alpha_{n-1}(1)$ is the
desired number. 
\end{proof}
\begin{rem}
\label{rem:Hamel basis}  It was G. Hamel  who first succeeded in
proving the existence of discontinuous additive functions.  In fact,
he proved the following theorem in the case when $N=1$:
\begin{thm*}[{\cite[Thm 5.2.2]{Kuczma2009Functional}}]
 Any function $g:H\rightarrow\mathbb{R}$ from an arbitrary Hamel
basis $H$  of the $\mathbb{Q}$-vector space $\mathbb{R}^{N}$ to
$\mathbb{R}$ extends to a unique additive function $f:\mathbb{R}^{N}\to\mathbb{R}$
such that $f\restriction_{H}=g$.
\end{thm*}
In fact, every (discontinuous) additive function $f:\mathbb{R}^{N}\to\mathbb{R}$
may be obtained in such a way.  Hence, there exist pathological polynomial
maps $\mathbb{R}_{\ge0}^{N}\to\mathbb{R}$ if one does not insist
on the continuity. 
\end{rem}
\begin{cor}
\label{cor:polynomial set in 1 var} The image of any nonconstant
continuous polynomial map $f:\mathbb{R}_{\ge0}\rightarrow\mathbb{R}$
is an unbounded interval. 
\end{cor}
\begin{proof}
By Theorem \ref{thm:polynomial in 1 var}, $f$ must be a polynomial
of degree $\ge1$, whose image $f(\mathbb{R}_{\ge0})$ is certainly
an unbounded interval in $\mathbb{R}$.
\end{proof}
\begin{cor}
\label{cor:vector of polynomials in 1 var} Every continuous polynomial
map $f:\mathbb{R}_{\ge0}\rightarrow\mathbb{R}^{M}$ of degree $\le d$
is a vector of polynomials of degree $\le d$.
\end{cor}
\begin{proof}
 By Theorem \ref{thm:polynomial in 1 var}, the induced map $f_{i}:=\pi_{i}\circ f$,
where $\pi_{i}:\mathbb{R}^{M}\to\mathbb{R}$ be the projection map
of the $i$th coordinates, is a continuous polynomial map of degree
$\le d$. 
\end{proof}
Before stating another theorem, we record some properties about the
group $\mathcal{U}_{n}(\mathbb{R})$ of upper unitriangular $n\times n$
matrices over $\mathbb{R}$ in the lemma below. Let $E_{i,j}$ be
the $n\times n$ matrix with all entries $0$ except the $(i,j)$-entry
$1$, and $I$ be the identity $n\times n$ matrix, and $T_{i,j}(a)$
be the unitriangular matrix of the form $I+aE_{i,j}$ for $i\ne j$
and $a\in\mathbb{R}$. 
\begin{lem}
\label{lem:properties of unitriangular matrices} For any $T\in\mathcal{U}_{n}(\mathbb{R})$,
$T$ can be written as $I+T_{u}$, where $T_{u}$ is strictly upper
triangular and thus nilpotent with index $\le n$, i.e., $T_{u}^{i}=0$
for all $i\ge n$. 
\begin{enumerate}
\item \label{enu:inverse of T} The inverse of $T$ is given by $T^{-1}=I+\sum_{i=1}^{n-1}(-T_{u})^{i}$. 
\item For distinct $i,j\in\{1,\ldots,n\}$ and $a\in\mathbb{R}$, one has
$T_{i,j}(a)^{-1}=T_{i,j}(-a)$. 
\item \label{enu:distinct i,j,l} For distinct $i,j,l\in\{1,\ldots,n\}$
and $a,b\in\mathbb{R}$, one has $[T_{i,j}(a),T_{j,l}(b)]=T_{i,l}(ab)$. 
\item For distinct $i,j,k,l\in\{1,\ldots,n\}$ and $a,b\in\mathbb{R}$,
one has $[T_{i,j}(a),T_{k,l}(b)]=I$. 
\end{enumerate}
\end{lem}
\begin{proof}
We have the following equation of power series, 
\[
(I+T_{u})^{-1}=\sum_{i=0}^{\infty}(-T_{u})^{i},
\]
from which (\ref{enu:inverse of T}) follows easily, since the right
hand side ends in finitely many steps, while the other statements
follow easily from (\ref{enu:inverse of T}) and the formula of matrix
products 
\[
E_{i,j}E_{k,l}=\delta_{j,k}E_{i,l}=\begin{cases}
E_{i,l}, & \text{if }j=k,\\
0, & \text{otherwise}.
\end{cases}
\]
\end{proof}
For all $1\le k\le n$, let $\mathcal{U}_{n,k}(\mathbb{R})$ be the
subset of $\mathcal{U}_{n}(\mathbb{R})$ consisting of matrices $(t_{i,j})$
such that $t_{i,j}=\delta_{i,j}$ for $j<i+k$, and for all $1\le k\le n-1$,
define maps $\phi_{k}$ in the following way 
\[
\phi_{k}:\mathcal{U}_{n,k}(\mathbb{R})\to(\mathbb{R}^{n-k},+);\quad T=(t_{i,j})\mapsto(t_{1,1+k},t_{2,2+k},\ldots,t_{n-k,n}).
\]
For convenience sake, we call $(t_{1,1+k},t_{2,2+k},\ldots,t_{n-k,n})$
the $k$th diagonal entries of $T$.  Note that $\mathcal{U}_{n,k}(\mathbb{R})$
is a subgroup of $\mathcal{U}_{n}(\mathbb{R})$ with $\mathcal{U}_{n,1}(\mathbb{R})=\mathcal{U}_{n}(\mathbb{R})$
and $\mathcal{U}_{n,n}(\mathbb{R})=\{I\}$. 
\begin{lem}
\label{lem:properties of U_n(R)} The following properties hold: 
\begin{enumerate}
\item \label{enu:phi_k homomorphism} The map $\phi_{k}$ is a homomorphism
of groups with kernel $\mathcal{U}_{n,k+1}(\mathbb{R})$. 
\item The derived group $C^{2}\mathcal{U}_{n,k}(\mathbb{R})$ is a subgroup
of $\mathcal{U}_{n,k+1}(\mathbb{R})$. 
\item The subgroup $\mathcal{U}_{n,k+1}(\mathbb{R})$ is a normal subgroup
of $\mathcal{U}_{n,k}(\mathbb{R})$ for all $k\ge1$. 
\item \label{enu:U_=00007Bn,k=00007D(R) generated by T_=00007Bi,j=00007D(a)}
The subgroup $\mathcal{U}_{n,k}(\mathbb{R})$ is generated by $\{T_{i,j}(a)\mid a\in\mathbb{R},j\ge i+k\}$. 
\item For all $k\ge1$, $C^{k}\mathcal{U}_{n}(\mathbb{R})$ is a subgroup
of $\mathcal{U}_{n,k}(\mathbb{R})$ and $\mathcal{U}_{n}(\mathbb{R})$
is nilpotent of class $\le n-1$. 
\end{enumerate}
\end{lem}
\begin{proof}
See \cite[Exercise 13.38]{Drutu2018Geometric}. 
\end{proof}
The last statement in the previous lemma can be strengthened as follows: 
\begin{cor}
\label{cor:C^k(U_n(R))=00003DU_=00007Bn,k=00007D(R)} For all $k\ge1$,
$C^{k}\mathcal{U}_{n}(\mathbb{R})$ is the group $\mathcal{U}_{n,k}(\mathbb{R})$
and $\mathcal{U}_{n}(\mathbb{R})$ is nilpotent of class $n-1$. 
\end{cor}
\begin{proof}
 One proves by induction on $k$ that $T_{i,j}(a)\in C^{k}\mathcal{U}_{n}(\mathbb{R})$
for $j\ge i+k$ and $a\in\mathbb{R}$. 
\end{proof}
\begin{thm}
\label{thm:unitriangular matrix of polynomials in 1 var} Let $f_{i,j}:\mathbb{R}_{\ge0}\to\mathbb{R}$
with $1\le i<j\le n$ be continuous polynomial maps of degree $d_{i,j}$
and $f:\mathbb{R}_{\ge0}\rightarrow\mathcal{U}_{n}(\mathbb{R})$ be
a function whose matrix form is 
\begin{equation}
\begin{pmatrix}1 & f_{1,2} & f_{1,3} & \cdots & f_{1,n}\\
 & 1 & f_{2,3} & \cdots & f_{2,n}\\
 &  & 1 & \ddots & \vdots\\
 &  &  & \ddots & f_{n-1,n}\\
 &  &  &  & 1
\end{pmatrix}.\label{eq:matrix form}
\end{equation}
Then, the function $f$ is a continuous polynomial map. 

Conversely, every continuous \footnote{The real Lie group $\mathcal{U}_{n}(\mathbb{R})$ is given the usual
metric topology. } polynomial map $f:\mathbb{R}_{\ge0}\rightarrow\mathcal{U}_{n}(\mathbb{R})$
is of this form. 
\end{thm}
\begin{proof}
 Notice that $f$ is given by the following elementwise product with
a particular order 
\[
f=\prod_{i=1}^{n-1}\prod_{j=i+1}^{n}(I+f_{i,j}E_{i,j}).
\]
By Theorem \ref{thm:product of two polynomial maps}, it suffices
to show that each $I+f_{i,j}E_{i,j}$ is a polynomial map. This follows
easily from the assumption that each $f_{i,j}$ is a continuous polynomial
map. By Theorem \ref{thm:polynomial in 1 var}, $f_{i,j}:\mathbb{R}_{\ge0}\to\mathbb{R}$
are polynomials of degree $d_{i,j}$ for all $1\le i<j\le n$. Hence,
it follows that $f$ is continuous. 

Conversely, suppose that $f$ has lc-height $k$, i.e., the image
of $f$ lies in the subgroup $C^{k}\mathcal{U}_{n}(\mathbb{R})=\mathcal{U}_{n,k}(\mathbb{R})$
for $1\le k\le n$. The proof for the converse statement is by downward
induction on $k$. If $k=n$, then $\mathcal{U}_{n,k}(\mathbb{R})=\{I\}$
and $f_{i,j}$ are zero polynomials. If $k=n-1$, then all $f_{i,j}$
are zero polynomials except $f_{1,n}$. By (\ref{enu:phi_k homomorphism})
in Lemma \ref{lem:properties of U_n(R)}, $\phi_{n-1}$ is a homomorphism
(in fact an isomorphism). Since $f$ is a continuous polynomial map,
so is $f_{1,n}$. By Theorem \ref{thm:polynomial in 1 var}, $f_{1,n}$
is a polynomial. 

Suppose we have shown this for continuous polynomial maps of lc-height
$>2$ and $f$ is a continuous polynomial map of lc-height $1$. Then,
we apply the group homomorphism $\phi_{1}$ to $f(t)$ to pick out
the first diagonal entries $(f_{1,2},f_{2,3},\ldots,f_{n-1,n})$.
Then, $(f_{1,2},f_{2,3},\ldots,f_{n-1,n}):\mathbb{R}_{\ge0}\to\mathbb{R}^{n-1}$
is a continuous polynomial map. By Corollary \ref{cor:vector of polynomials in 1 var},
$f_{i,i+1}:\mathbb{R}_{\ge0}\to\mathbb{R}$ are polynomials for all
$1\le i\le n-1$. Next, we multiply $h=I-\sum_{i=1}^{n-1}f_{i,i+1}(t)E_{i,i+1}$
on the left hand side of $f$ and obtain: 
\begin{align*}
hf & =\begin{pmatrix}1 & -f_{1,2}\\
 & 1 & -f_{2,3}\\
 &  & 1 & \ddots\\
 &  &  & \ddots & -f_{n-1,n}\\
 &  &  &  & 1
\end{pmatrix}\begin{pmatrix}1 & f_{1,2} & f_{1,3} & \cdots & f_{1,n}\\
 & 1 & f_{2,3} & \cdots & f_{2,n}\\
 &  & 1 & \ddots & \vdots\\
 &  &  & \ddots & f_{n-1,n}\\
 &  &  &  & 1
\end{pmatrix}\\
 & =\begin{pmatrix}1 & 0 & f_{1,3}-f_{1,2}f_{2,3} & f_{1,4}-f_{1,2}f_{2,4} & \cdots & f_{1,n}-f_{1,2}f_{2,n}\\
 & 1 & 0 & f_{2,4}-f_{2,3}f_{3,4} & \cdots & f_{2,n}-f_{2,3}f_{3,n}\\
 &  & 1 & 0 & \ddots & \vdots\\
 &  &  & 1 & \ddots & f_{n-2,n}-f_{n-2,n-1}f_{n-1,n}\\
 &  &  &  & \ddots & 0\\
 &  &  &  &  & 1
\end{pmatrix}.
\end{align*}
Then, $hf$ is a continuous polynomial map of lc-height $\ge2$. By
the induction hypotheses, each entry of $hf$ is a polynomial. Then,
the second diagonal entries 
\[
f_{i,i+2}=(f_{i,i+2}-f_{i,i+1}f_{i+1,i+2})+f_{i,i+1}f_{i+1,i+2}
\]
are polynomials, since $f_{i,i+1}f_{i+1,i+2}$ are already known to
be polynomials. Inductively, we can show that all $j$th diagonal
entries $f_{i,j}=(f_{i,j}-f_{i,i+1}f_{i+1,j})+f_{i,i+1}f_{i+1,j}$
are polynomials. 
\end{proof}
Next, we will talk about polynomial maps in several variables. Any
usual polynomial $f:\mathbb{R}_{\ge0}^{N}\rightarrow\mathbb{R}$ in
$N$ variables of degree $\le d$ is a polynomial map of degree $\le d$
in our sense. The converse also holds, provided that the polynomial
map $f:\mathbb{R}_{\ge0}^{N}\rightarrow\mathbb{R}$ is assumed to
be continuous. 
\begin{thm}
\label{thm:polynomial in N var} Every continuous polynomial map $f:\mathbb{R}_{\ge0}^{N}\rightarrow\mathbb{R}$
of degree $d$ is a polynomial in $N$ variables of degree $d$. 
\end{thm}
\begin{proof}
We proceed by induction on $N$ with the base case $N=1$ being Theorem
\ref{thm:polynomial in 1 var}. For any $\mathbf{a}\in\mathbb{R}_{\ge0}^{N-1}$
define $f_{\mathbf{a}}(t)=f(\mathbf{a},t)$. Then, $f_{\mathbf{a}}(t)$
is a continuous polynomial map of degree $\le d$ in the variable
$t$. By Theorem \ref{thm:polynomial in 1 var}, we can write $f_{\mathbf{a}}(t)=\sum_{i=0}^{d}c_{i}(\mathbf{a})t^{i}$,
where $c_{i}:\mathbb{R}_{\ge0}^{N-1}\to\mathbb{R}$ are continuous
functions.  Applying the finite forward difference operator of the
form $D_{(x_{1},\ldots,x_{N-1},0)}$ $d+1$ times to
\[
f(t_{1},\ldots,t_{N-1},t_{N})=\sum_{i=0}^{d}c_{i}(t_{1},\ldots,t_{N-1})t^{i},
\]
we have $0=\sum_{i=0}^{d}D^{d+1}c_{i}(t_{1},\ldots,t_{N-1})t_{N}^{i}$.
By the linear independence of geometric progressions $t_{N}^{0},t_{N}^{1},\ldots,t_{N}^{d}$,
$c_{i}(t_{1},\ldots,t_{N-1})$ are continuous polynomial maps of degree
$\le d$. By the induction hypothesis, they are given by polynomials
in the usual sense, so the same is true for $f$. 
\end{proof}
\begin{cor}
\label{cor:polynomial set in N var} The image of any nonconstant
continuous polynomial map $f:\mathbb{R}_{\ge0}^{N}\rightarrow\mathbb{R}$
is an unbounded interval. 
\end{cor}
\begin{proof}
By Theorem \ref{thm:polynomial in N var}, $f$ must be a polynomial
of degree $\ge1$, whose image $f(\mathbb{R}_{\ge0}^{N})$ is certainly
an unbounded interval in $\mathbb{R}$. Alternatively, one may argue
with the help of the induced continuous polynomial map $\mathbb{R}_{\ge0}\hookrightarrow\mathbb{R}_{\ge0}^{N}\to\mathbb{R}$
and Theorem \ref{thm:polynomial in 1 var}. 
\end{proof}
\begin{cor}
\label{cor:vector of polynomials in N var} Every continuous polynomial
map $f:\mathbb{R}_{\ge0}^{N}\to\mathbb{R}^{M}$ of degree $\le d$
is vector of polynomials in $N$ variables of degree $\le d$. 
\end{cor}
\begin{proof}
The proof is the same as the one of Corollary \ref{cor:vector of polynomials in 1 var}. 
\end{proof}
\begin{thm}
\label{thm:unitriangular matrix of polynomials in N var} Let $f_{i,j}:\mathbb{R}_{\ge0}^{N}\to\mathbb{R}$
with $1\le i<j\le n$ be continuous polynomial maps of degree $\le d_{i,j}$
and $f:\mathbb{R}_{\ge0}^{N}\rightarrow\mathcal{U}_{n}(\mathbb{R})$
be a function with matrix form given by (\ref{eq:matrix form}). Then,
the function $f$ is a continuous polynomial map. 

Conversely, every continuous polynomial map $f:\mathbb{R}_{\ge0}^{N}\rightarrow\mathcal{U}_{n}(\mathbb{R})$
is of this form. 
\end{thm}
\begin{proof}
The proof is the same as the one of Theorem \ref{thm:unitriangular matrix of polynomials in 1 var},
except that one needs to replace Theorem \ref{thm:polynomial in 1 var}
by Theorem \ref{thm:polynomial in N var}, and Corollary \ref{cor:vector of polynomials in 1 var}
by Corollary \ref{cor:vector of polynomials in N var}. 
\end{proof}

\section{Estimation of the Degree \label{sec:Degree}}

The most important quantity of a polynomial map is its (lc-)degree.
The first attempt has been given in Corollary \ref{cor:lc-degree <=00003D a superadditive vector}
via superadditive vectors, but it is not good enough. So we will try
to estimate them by working out a formula for the lower and upper
bounds of the (lc-)degree, in particular, of polynomial maps of the
form $\mathbb{R}_{\ge0}^{N}\to\mathcal{U}_{n}(\mathbb{R})$. 

Notice that in the first part of Theorem \ref{thm:unitriangular matrix of polynomials in 1 var},
we do not give any information about the degree $d$ of the polynomial
map $f$, which should be closely related to the degrees $d_{i,j}$
for all $1\le i<j\le n$. Notice that $I+f_{i,i+k}E_{i,i+k}$ is a
polynomial map of lc-degree 
\[
(-\infty,\cdots,-\infty,d_{i,i+k},\ldots,d_{i,i+k})\in\mathbb{Z}_{*}^{n-1},
\]
where the first $d_{i,i+k}$ appears in the $k$th entry, as its image
lies in $\mathcal{U}_{n,k}(\mathbb{R})=C^{k}\mathcal{U}_{n}(\mathbb{R})$.
Then, there exists a least superadditive vector above all of the lc-degree,
which should be an upper bound of the (lc)-degree of $f$ in view
of Corollary \ref{cor:lc-degree <=00003D a superadditive vector}.
But this estimation is quite coarse. 

For a better understanding, two motivating examples are provided.
They are the continuous Heisenberg group $H_{3}(\mathbb{R})=\mathcal{U}_{3}(\mathbb{R})$
and the nilpotent Lie group $\mathcal{U}_{4}(\mathbb{R})$. 
\begin{example}
Let $f_{i,j}:\mathbb{R}_{\ge0}\to\mathbb{R}$ be continuous polynomial
maps of degree $d_{i,j}$ and set 
\[
f_{3}=:\begin{pmatrix}1 & f_{1,2} & f_{1,3}\\
0 & 1 & f_{2,3}\\
0 & 0 & 1
\end{pmatrix}:\mathbb{R}_{\ge0}\rightarrow H_{3}(\mathbb{R}),\qquad f_{4}:=\begin{pmatrix}1 & f_{1,2} & f_{1,3} & f_{1,4}\\
0 & 1 & f_{2,3} & f_{2,4}\\
0 & 0 & 1 & f_{3,4}\\
0 & 0 & 0 & 1
\end{pmatrix}:\mathbb{R}_{\ge0}\rightarrow\mathcal{U}_{4}(\mathbb{R}).
\]
Then, $f_{3}$ is a continuous polynomial map of degree $\le\max\{d_{1,3},d_{1,2}+d_{2,3}\}$
and $\ge\max\{f_{1,2},f_{2,3}\}$, and $f_{4}$ is a continuous polynomial
map of degree $\le\max\{d_{1,4},d_{1,2}+d_{2,4},d_{1,3}+d_{3,4},d_{1,2}+d_{2,3}+d_{3,4}\}$
and $\ge\max\{d_{1,2},d_{2,3},d_{3,4}\}$. 
\end{example}
\begin{thm}
\label{thm: bound degree 1} Let $f$ be as in Theorem \ref{thm:unitriangular matrix of polynomials in 1 var}.
Then, $f$ is a polynomial map of degree bounded below by 
\begin{equation}
\max\{d_{k,k+1}\mid1\le k\le n-1\},\label{eq:lower bound degree}
\end{equation}
and bounded above by 
\begin{equation}
\max\left\{ d_{k_{1},k_{2}}+\cdots+d_{k_{n-1},k_{n}}\mid1=k_{1}\le k_{2}\le\cdots\le k_{n-1}\le k_{n}=n\right\} ,\label{eq:upper bound degree}
\end{equation}
where $d_{i,j}$ is defined to be $0$ if $i=j$. 
\end{thm}
\begin{proof}
Denote $f_{i,j}(s+t)$ by $s_{i,j}$ and $f_{i,j}(t)$ by $t_{i,j}$,
set $s_{i,i}=t_{i,i}=1$, and write $S=I+S_{u}=f(s+t)$ and $T=I+T_{u}=f(t)$,
where $S_{u}$ (resp. $T_{u}$) has entries given by $s_{i,j}$ (resp.
$t_{i,j}$). 

Suppose that $f$ is a polynomial map of degree $d$. By (\ref{enu:phi_k homomorphism})
in Lemma \ref{lem:properties of U_n(R)}, $\phi_{1}$ is a homomorphism.
By Proposition \ref{prop:direct sum}, the induced polynomial map
$\phi_{1}\circ f=(f_{1,2},\ldots,f_{n-1,n}):\mathbb{R}_{\ge0}\rightarrow(\mathbb{R}^{n-1},+)$
is a polynomial of degree $\max\{d_{k,k+1}\mid1\le k\le n-1\}$, which
by Proposition \ref{prop:homomorphism of groups} is at most $d$. 

The proof for the upper bound is given by induction on $d$. The case
when $d\le0$ is trivial.  For the induction step, in view of (\ref{enu:inverse of T})
in Lemma \ref{lem:properties of U_n(R)}, we see that $L:=f(s+t)f^{-1}(t)$
is given by  
\begin{align*}
(I+S_{u})(I+T_{u})^{-1} & =I+\sum_{i=1}^{n}(-1)^{i}(T_{u}-S_{u})T_{u}^{i-1},
\end{align*}
and that $R:=f(t)^{-1}f(s+t)$ is given by  
\begin{align*}
(I+T_{u})^{-1}(I+S_{u}) & =I+\sum_{i=1}^{n}(-1)^{i}T_{u}^{i-1}(T_{u}-S_{u}).
\end{align*}
Since each entry of the matrix $L$ (resp. $R$) is a polynomial,
the crux of the proof is to find its expression and estimate the upper
bound of its degree. 

Notice that these two equations are similar to the inverse of $T$
as in Lemma \ref{lem:properties of U_n(R)}, except that one replaces
the first (resp. last) $T_{u}$ by $T_{u}-S_{u}$. Since $T_{u}$
is strictly upper triangular and nilpotent with index $\le n$, one
could easily write down the expression of each entry of $T^{-1}$.
Indeed, for example, the first diagonal entries of $T^{-1}$ are given
by $(-t_{1,2},-t_{2,3},\ldots,-t_{n-1,n})$, and the second diagonal
entries of $T^{-1}$ are given by $(-t_{1,3}+t_{1,2}t_{2,3},-t_{2,4}+t_{2,3}t_{3,4},\ldots,-t_{n-2,n}+t_{n-2,n-1}t_{n-1,n})$.
In general, the $(i,j)$-entry in the $(j-i)$th diagonal entries
of $T^{-1}$ is given by 
\[
-t_{i,j}+\sum_{i<k<j}t_{i,k}t_{k,j}-\sum_{i<k_{1}<k_{2}<j}t_{i,k_{1}}t_{k_{1},k_{2}}t_{k_{2},j}+\cdots.
\]
Similarly, the first diagonal entries of $L$ and $R$ are given by
$s_{i,i+1}-t_{i,i+1}$, and the second diagonal entries of $L$ (resp.
$R$) are given by 
\[
(-(t_{1,3}-s_{1,3})+(t_{1,2}-s_{1,2})t_{2,3},\ldots,s_{n-2,n}-t_{n-2,n}+(t_{n-2,n-1}-s_{n-2,n-1})t_{n-1,n}),
\]
\[
\text{(resp. }(-(t_{1,3}-s_{1,3})+t_{1,2}(t_{2,3}-s_{2,3}),\ldots,s_{n-2,n}-t_{n-2,n}+t_{n-2,n-1}(t_{n-1,n}-s_{n-1,n}))\text{).}
\]
In general, the $(i,j)$-entry in the $(j-i)$th diagonal entries
$L$ (resp. $R$) is given by 
\[
-(t_{i,j}-s_{i,j})+\sum_{i<k<j}(t_{i,k}-s_{i,k})t_{k,j}-\sum_{i<k_{1}<k_{2}<j}(t_{i,k_{1}}-s_{i,k_{1}})t_{k_{1},k_{2}}t_{k_{2},j}+\cdots,
\]
\[
\left(\text{resp. }-(t_{i,j}-s_{i,j})+\sum_{i<k<j}t_{i,k}(t_{k,j}-s_{k,j})-\sum_{i<k_{1}<k_{2}<j}t_{i,k_{1}}t_{k_{1},k_{2}}(t_{k_{2},j}-s_{k_{2},j})+\cdots\right).
\]
Then, the degree of the $(i,j)$-entry in the $(j-i)$th diagonal
entries of $L$ and $R$ is 
\[
\le e_{i,j}:=\max_{i=k_{1}\le k_{2}\le\cdots\le k_{j-i}\le k_{j-i+1}=j}\left\{ d_{k_{1},k_{2}}+\cdots+d_{k_{j-i},k_{j-i+1}}\right\} -1,
\]
which is $1$ less than the degree of the $(i,j)$-entry in the $(j-i)$th
diagonal entries of $T^{-1}$. 

Since $L$ and $R$ are polynomial maps of degree $d-1$, by the induction
hypothesis, we can apply the upper bounds (\ref{eq:upper bound degree})
to $L$ and $R$, and obtain that 
\begin{align*}
d-1\le & \max\left\{ e_{j_{1},j_{2}}+\cdots+e_{j_{n-1},j_{n}}\mid1=j_{1}\le j_{2}\le\cdots\le j_{n-1}\le j_{n}=n\right\} \\
\le & \max_{1=j_{1}\le j_{2}\le\cdots\le j_{n-1}\le j_{n}=n}\left\{ \max_{j_{1}=k_{1}\le k_{2}\le\cdots\le k_{j_{2}-1}\le k_{j_{2}}=j_{2}}\left\{ d_{k_{1},k_{2}}+\cdots+d_{k_{j_{2}-1},k_{j_{2}}}\right\} -1+\cdots\right.\\
 & \left.+\max_{j_{n-1}=k_{1}\le k_{2}\le\cdots\le k_{j_{n}-j_{n-1}}\le k_{j_{n}-j_{n-1}+1}=j_{n}}\left\{ d_{k_{1},k_{2}}+\cdots+d_{k_{j_{n}-j_{n-1}},k_{j_{n}-j_{n-1}+1}}\right\} -1\right\} \\
\le & \max\left\{ d_{k_{1},k_{2}}+\cdots+d_{k_{n-1},k_{n}}\mid1=k_{1}\le k_{2}\le\cdots\le k_{n-1}\le k_{n}=n\right\} -1.
\end{align*}
Therefore, we have $d\le\max\{d_{k_{1},k_{2}}+\cdots+d_{k_{n-1},k_{n}}\mid1=k_{1}\le k_{2}\le\cdots\le k_{n-1}\le k_{n}=n\}$. 
\end{proof}
\begin{rem*}
Generically, the degree $d$ should achieve this upper bound. But
as one can see from the proof, there is a rare possibility that some
polynomials might cancel with each other when the coefficients of
these polynomials satisfy a certain system of nontrivial polynomial
equations, so that is why inequality (\ref{eq:upper bound degree})
only gives an upper bound. 
\end{rem*}
 Here is an example in which the actual degree of the polynomial
map can be much smaller than this upper bound. A one-parameter subgroup
in a topological group $G$ is a continuous group homomorphism $\varphi:(\mathbb{R},+)\rightarrow G$.
 If $\varphi$ is injective, then the image $\varphi(\mathbb{R})$
will be a subgroup isomorphic to $\mathbb{R}$ as an additive group.
Typically, trivial homomorphisms are not considered to be one-parameter
subgroups. So one-parameter subgroups are polynomial maps of degree
$1$. 
\begin{example*}
 Then, $f(t)=e^{tA}:\mathbb{R}\to\mathcal{U}_{n}(\mathbb{R})$, where
 $A$ is a strictly upper triangular $n\times n$ matrix in $M_{n}(\mathbb{R})$,
is a one-parameter subgroup in $\mathcal{U}_{n}(\mathbb{R})$. For
each $1\le k\le n-1$, the $k$th diagonal entries $f_{i,i+k}$ are
polynomials of degree $\le d_{i,i+k}=k$. Hence, the inequality (\ref{eq:upper bound degree})
yields an upper bound  
\[
\max\left\{ d_{k_{1},k_{2}}+\cdots+d_{k_{n-1},k_{n}}\mid1=k_{1}\le k_{2}\le\cdots\le k_{n-1}\le k_{n}=n\right\} =n-1.
\]
\end{example*}
As mentioned before, the degree of the product of two polynomial maps
is quite mysterious in Theorem \ref{thm:product of two polynomial maps}.
But with the help of Theorem \ref{thm: bound degree 1}, we can say
more about this. 
\begin{cor}
\label{cor:upper bound product} Let $f,f':\mathbb{R}_{\ge0}\to\mathcal{U}_{n}(\mathbb{R})$
be two polynomial maps of degree $\le d$ and $\le d'$ respectively,
where $d$ and $d'$ are the upper bounds given by (\ref{eq:upper bound degree}).
Then, the product $ff':\mathbb{R}_{\ge0}\to\mathcal{U}_{n}(\mathbb{R})$
is a polynomial map of degree $\le d+d'$.
\end{cor}
\begin{proof}
The $(i,j)$-entry of $f(t)f'(t)$ is given by $\sum_{i\le k\le j}f_{i,k}(t)f_{k,j}'(t)$,
which has degree 
\[
\le e_{i,j}:=\max\{d_{i,k}+d_{k,j}'\mid i\le k\le j\}.
\]
Hence, by Theorem \ref{thm: bound degree 1}, $ff'$ is a polynomial
map of degree $\le$ 
\begin{align*}
 & \max\left\{ e_{k_{1},k_{2}}+\cdots+e_{k_{n-1},k_{n}}\mid1=k_{1}\le k_{2}\le\cdots\le k_{n-1}\le k_{n}=n\right\} \\
\le & \max_{1=k_{1}\le k_{2}\le\cdots\le k_{n-1}\le k_{n}=n}\left\{ \max_{k_{1}\le l_{1}\le k_{2}}\{d_{k_{1},l_{1}}+d_{l_{1},k_{2}}'\}+\cdots+\max_{k_{n-1}\le l_{n-1}\le k_{n}}\{d_{k_{n-1},l_{n-1}}+d_{l_{n-1},k_{n}}'\}\right\} \\
\le & \max_{1=k_{1}\le k_{2}\le\cdots\le k_{n-1}\le k_{n}=n}\left\{ d_{k_{1},k_{2}}+\cdots+d_{k_{n-1},k_{n}}\right\} +\max_{1=k_{1}\le k_{2}\le\cdots\le k_{n-1}\le k_{n}=n}\left\{ d_{k_{1},k_{2}}'+\cdots+d_{k_{n-1},k_{n}}'\right\} ,
\end{align*}
where the last inequality holds for the following reason. For each
possible choice of $l_{1},\ldots,l_{n-1}$, such that $1=k_{1}\le l_{1}\le k_{2}\le l_{2}\le k_{3}<\ldots<k_{n-1}\le l_{n-1}\le k_{n}=n$,
we have 
\begin{align*}
 & d_{k_{1},l_{1}}+d_{l_{1},k_{2}}'+\cdots+d_{k_{n-1},l_{n-1}}+d_{l_{n-1},k_{n}}'=d_{k_{1},l_{1}}+\cdots+d_{k_{n-1},l_{n-1}}+d_{l_{1},k_{2}}'+\cdots+d_{l_{n-1},k_{n}}'\\
\\
\le & d_{k_{1},l_{1}}+d_{l_{1},k_{2}}+\cdots+d_{k_{n-1},l_{n-1}}+d_{l_{n-1},k_{n}}+d_{k_{1},l_{1}}'+d_{l_{1},k_{2}}'+\cdots+d_{k_{n-1},l_{n-1}}'+d_{l_{n-1},k_{n}}'\\
\le & \max_{1=k_{1}\le k_{2}\le\cdots\le k_{n-1}\le k_{n}=n}\left\{ d_{k_{1},k_{2}}+\cdots+d_{k_{n-1},k_{n}}\right\} +\max_{1=k_{1}\le k_{2}\le\cdots\le k_{n-1}\le k_{n}=n}\left\{ d_{k_{1},k_{2}}'+\cdots+d_{k_{n-1},k_{n}}'\right\} .
\end{align*}
Hence, the proof is complete, since the last row is nothing but $d+d'$. 
\end{proof}
 In the same manner, we can talk more about the lc-degree of $f$.

\begin{thm}
\label{thm: bound lc-degree 1} Let $f$ be as in Theorem \ref{thm:unitriangular matrix of polynomials in 1 var}
and $\hat{d}=(d_{1},d_{2},\ldots,d_{n-1})$ be the lc-degree of $f$.
Then, 
\begin{equation}
\begin{cases}
-\infty\le d_{1}=\max\{d_{k,k+1}\mid1\le k\le n-1\},\\
d_{i-1}\le d_{i}\le\max\left\{ d_{k_{1},k_{2}}+\cdots+d_{k_{i},k_{i+1}}\right|\\
\left.\qquad\qquad\qquad\quad1\le k_{1}\le k_{2}\le\cdots\le k_{i}\le k_{i+1}=k_{1}+i\le n\right\} , & 2\le i\le n-1,
\end{cases}\label{eq:bound lc-degree}
\end{equation}
where $d_{i,j}$ is defined to be $0$ if $i=j$. In particular,
when $i=n-1$, we obtain the same upper bound for the degree of $f$
as in Theorem \ref{thm: bound degree 1}. 
\end{thm}
\begin{proof}
By definition, $d_{1}$ is the degree of $f\mod C^{2}\mathcal{U}_{n}(\mathbb{R})$,
which by Corollary \ref{cor:C^k(U_n(R))=00003DU_=00007Bn,k=00007D(R)}
is the same as the degree of $f\mod\mathcal{U}_{n,2}(\mathbb{R})$,
and the same as the degree of 
\[
\phi_{1}\circ f=(f_{1,2},\ldots,f_{n-1,n}):\mathbb{R}_{\ge0}\rightarrow(\mathbb{R}^{n-1},+).
\]
We know that $\phi_{1}\circ f$ is a polynomial of degree $\max\{d_{k,k+1}\mid1\le k\le n-1\}$. 

For $2\le i\le n-1$, the proof is almost the same as the one of Theorem
\ref{thm: bound degree 1}, except that instead of calculating the
degree of $f$, one calculates the degree of $f\mod C^{i+1}\mathcal{U}_{n}(\mathbb{R})=f\mod\mathcal{U}_{n,i+1}(\mathbb{R})$.
It is easy to see why only the $\le i$th diagonal terms are involved
in the inequalities (\ref{eq:bound lc-degree}). 
\end{proof}
\begin{rem*}
The lc-degree $\hat{d}$ should generically achieve this upper bound,
but there is a rare possibility that it is strictly less than that.
The one-parameter subgroups in $\mathcal{U}_{n}(\mathbb{R})$ provide
such examples. 
\end{rem*}
 Of course, the above results generalize to polynomial maps in multivariate
cases. 
\begin{thm}
\label{thm:bound degree N} Let $f$ be as in Theorem \ref{thm:unitriangular matrix of polynomials in N var}.
Then, $f$ is a polynomial map of degree bounded below by 
\begin{equation}
\max\{d_{k,k+1}\mid1\le k\le n-1\},\label{eq:lower bound degree N}
\end{equation}
and bounded above by 
\begin{equation}
\le\max\left\{ d_{k_{1},k_{2}}+\cdots+d_{k_{n-1},k_{n}}\mid1=k_{1}\le k_{2}\le\cdots\le k_{n-1}\le k_{n}=n\right\} ,\label{eq:upper bound N}
\end{equation}
where $d_{i,j}$ is defined to be $0$ if $i=j$. 
\end{thm}
\begin{proof}
The proof is the same as the one of Theorem \ref{thm: bound degree 1}. 
\end{proof}
\begin{cor}
Let $f,f':\mathbb{R}_{\ge0}^{N}\to\mathcal{U}_{n}(\mathbb{R})$ be
two polynomial maps of degree $\le d$ and $\le d'$ respectively,
where $d$ and $d'$ are the upper bounds given by (\ref{eq:upper bound N}).
Then, the product $ff':\mathbb{R}_{\ge0}^{N}\to\mathcal{U}_{n}(\mathbb{R})$
is a polynomial map of degree $\le d+d'$.
\end{cor}
\begin{proof}
The proof is the same as the one of Corollary \ref{cor:upper bound product}. 
\end{proof}
\begin{thm}
\label{thm:bound lc-degree N} Let $f$ be as in Theorem \ref{thm:unitriangular matrix of polynomials in N var}
and $\hat{d}=(d_{1},d_{2},\ldots,d_{n-1})$ be the lc-degree of $f$.
Then, 

\begin{equation}
\begin{cases}
-\infty\le d_{1}=\max\{d_{k,k+1}\mid1\le k\le n-1\},\\
d_{i-1}\le d_{i}\le\max\left\{ d_{k_{1},k_{2}}+\cdots+d_{k_{i},k_{i+1}}\right|\\
\left.\qquad\qquad\qquad\quad1\le k_{1}\le k_{2}\le\cdots\le k_{i}\le k_{i+1}=k_{1}+i\le n\right\} , & 2\le i\le n-1,
\end{cases}\label{eq:bound lc-degree N}
\end{equation}
where $d_{i,j}$ is defined to be $0$ if $i=j$. In particular, when
$i=n-1$, we obtain the same upper bound for the degree of $f$ as
in Theorem \ref{thm:bound degree N}. 
\end{thm}
\begin{proof}
The proof is the same as the one of Theorem \ref{thm:bound degree N}.
\end{proof}

\section{Polynomial Sequences \label{sec:Polynomial Sequences}}

In this section, we will concentrate on special polynomial maps of
the form $\mathbb{N}_{0}\to G$. 
\begin{defn}
A polynomial map $g$ from $\mathbb{N}_{0}$ to a group $G$ will
be called a polynomial sequence. 

By abuse of terminology, we often call $g_{0},g_{1},g_{2},\ldots$
a polynomial sequence in $G$, where $g_{i}:=g(i)$, $\forall i\in\mathbb{N}_{0}$,
and denote by $\langle g\rangle$ the subgroup of $G$ generated by
the polynomial sequence $g_{0},g_{1},g_{2},\ldots$. 
\end{defn}
\begin{rem*}
We can talk about polynomial subsequence of $g_{0},g_{1},g_{2},\ldots$
in the following sense: 
\begin{enumerate}
\item Since the translation $T_{s}(g)(t):=g(t+s)$ of a polynomial sequence
$g$ by $s\in\mathbb{N}_{0}$ is a polynomial sequence, $g_{s},g_{1+s},g_{2+s},\ldots$
can be viewed as a polynomial subsequence. 
\item We have the polynomial sequence $\mathbb{N}_{0}\xrightarrow{\hat{k}}\mathbb{N}_{0}\xrightarrow{g}G$
induced by a homomorphism $\hat{k}:\mathbb{N}_{0}\to\mathbb{N}_{0}$;
$t\mapsto kt$, where $k\in\mathbb{N}_{0}$. Then, $g_{k},g_{2k},\ldots$
can be viewed as a polynomial subsequence. 
\end{enumerate}
Notice that a polynomial subsequence has degree no larger than the
original degree. 
\end{rem*}
By Proposition \ref{prop:ploynomial map in locally nilpotent subgroup},
the subgroup generated by a polynomial sequence is always finitely
generated. Then, Theorem \ref{thm:product of two polynomial maps}
can be slightly generalized in the case of polynomial sequences. 
\begin{cor}
The product of two polynomial sequences $f,f':\mathbb{N}_{0}\to G$
in a locally nilpotent group $G$ is a polynomial sequence. 
\end{cor}
\begin{defn}
A sequence $g$ is called periodic, if there exists $P\in\mathbb{N}$
such that $g_{i+P}=g_{i}$, $\forall i\in\mathbb{N}_{0}$.
\end{defn}
\begin{prop}
\label{prop:periodic polynomial sequence} A polynomial sequence in
a finite group is always periodic.
\end{prop}
\begin{proof}
Let $G$ be a finite group and $g:\mathbb{N}_{0}\to G$ be any polynomial
sequence.  Let $|G|$ be the order of $G$. The proof is by induction
on the degree $d$ of the polynomial sequence. If $d\le0$, then $g$
is constant, and thus periodic. If $d=1$, then we have
\[
g_{i+P}=l_{1}g_{i+P-1}=\cdots=l_{1}^{P}g_{i},\qquad g_{i+P}=g_{i+P-1}r_{1}=\cdots=g_{i}r_{1}^{P}.
\]
A suitable $P$ (for example $|G|$) can be chosen so that $l_{1}^{P}=r_{1}^{P}$
is the identity of $G$. 

Suppose that we have proved this for all polynomial sequences of degree
$<d$ and $g$ is a polynomial sequence of degree $\le d$. Then,
we have 
\begin{align*}
g_{i+P} & =L_{1}(g)(i+P-1)g_{i+P-1}=\cdots=L_{1}(g)(i+P-1)\cdots L_{1}(g)(i)g_{i},\\
g_{i+P} & =g_{i+P-1}R_{1}(g)(i+P-1)=\cdots=g_{i}R_{1}(g)(i)\cdots R_{1}(g)(i+P-1).
\end{align*}
Since $L_{1}(g)$ and $R_{1}(g)$ are polynomials of degree $\le d-1$,
they are periodic polynomial sequences, say, of periods $L$ and $R$
respectively. Thus, for a certain natural number $P$ (for example,
$\lcm(L,R)$), $L_{1}(g)(i+P-1)\cdots L_{1}(g)(i)$ and $R_{1}(g)(i)\cdots R_{1}(g)(i+P-1)$
are constant for all $i\in\mathbb{N}_{0}$. If necessary, one replaces
$P$ by $P^{|G|}$ in order to assure that 
\[
1_{G}=L_{1}(g)(i+P-1)\cdots L_{1}(g)(i)=R_{1}(g)(i)\cdots R_{1}(g)(i+P-1).
\]
Hence, $g_{i+P}=g_{i}$ for all $i\in\mathbb{N}_{0}$, i.e., $g$
has period $P$. 
\end{proof}
\begin{thm}
\label{thm: polynomial sequence in Z, Q, R} Every polynomial sequence
$f:\mathbb{N}_{0}\rightarrow R$, where $R$ can be the ring $\mathbb{Z}$,
$\mathbb{Q}$, or $\mathbb{R}$, of degree $\le d$ is a polynomial
of degree $\le d$. 
\end{thm}
\begin{proof}
The proof is the same as the proof of Theorem \ref{thm:polynomial in 1 var}.
One simply ignores the continuity part and replaces $\mathbb{Q}_{\ge0}$
by $\mathbb{N}_{0}$ everywhere. 
\end{proof}
\begin{cor}
\label{cor:vector of polynomial sequences in Z^M, Q^M, R^M} Every
polynomial sequence $f:\mathbb{N}_{0}\rightarrow R^{M}$, where $R$
could be the ring $\mathbb{Z}$, $\mathbb{Q}$, or $\mathbb{R}$,
of degree $\le d$ is a vector of polynomials of degree $\le d$. 
\end{cor}
\begin{proof}
Let $\pi_{i}:\mathbb{R}^{M}\to\mathbb{R}$ be the projection map of
the $i$th coordinates. Then, $f_{i}:=\pi_{i}\circ f$ is a polynomial
sequence of degree $\le d$. Then, by Theorem \ref{thm: polynomial sequence in Z, Q, R},
$f=(f_{1},\ldots,f_{M})$ is a vector of polynomials of degree $\le d$.
\end{proof}
\begin{thm}
\label{thm: polynomial sequence in U_n(R)} Let $f_{i,j}:\mathbb{N}_{0}\to R$,
where $R$ could be the ring $\mathbb{Z}$, $\mathbb{Q}$, or $\mathbb{R}$,
be polynomial sequences of degree $\le d_{i,j}$ for all $1\le i<j\le n$
and $f:\mathbb{N}_{0}\rightarrow\mathcal{U}_{n}(R)$ be a function
with matrix form given by (\ref{eq:matrix form}). Then, $f$ is a
polynomial sequence. 

Conversely, every polynomial sequence $f:\mathbb{N}_{0}\rightarrow\mathcal{U}_{n}(R)$
is of this form. 
\end{thm}
\begin{proof}
The proof is the same as the proof of Theorem \ref{thm:polynomial in N var}.
One simply ignores the continuity part and replaces $\mathbb{Q}_{\ge0}$
by $\mathbb{N}_{0}$ everywhere. 
\end{proof}
The combination of the following theorems by Mal'tsev and Ado shows
that each finitely generated torsion-free nilpotent group embeds in
$\mathcal{U}_{n}(\mathbb{R})$ for some $n$:
\begin{thm*}[A. I. Mal'tsev \cite{Mal49b}]
\label{thm:uniform lattice}  Every finitely generated torsion-free
nilpotent group $\Gamma$ of class $k$ embeds as a uniform lattice
in a simply-connected nilpotent Lie group $N$ of class $k$. Furthermore,
the group $N$ and the embedding $\Gamma\to N$ are unique up to an
isomorphism.
\end{thm*}
\begin{thm*}[Ado-Engel theorem]
\label{thm:simply-connected nilpotent Lie}  Every simply-connected
nilpotent Lie group $N$ embeds into $\mathcal{U}_{n}(\mathbb{R})$
for some $n$. 
\end{thm*}
\begin{thm}
\label{thm:polynomial sequence repeats constant} In a finitely generated
torsion-free nilpotent group $G$, a polynomial sequence $g_{0},g_{1},g_{2},\ldots$
can repeat a value infinitely many times if and only if the sequence
is constant. 
\end{thm}
\begin{proof}
By theorems of Mal'tsev and Ado, there exists $n\in\mathbb{N}$ such
that $G$ embeds into $\mathcal{U}_{n}(\mathbb{R})$. Then, consider
the induced polynomial sequence $f:\mathbb{N}_{0}\xrightarrow{g}G\hookrightarrow\mathcal{U}_{n}(\mathbb{R})$.
By Theorem \ref{thm: polynomial sequence in U_n(R)}, $f$ can be
written as an upper unitriangular matrix form with  polynomials $f_{i,j}:\mathbb{N}_{0}\to\mathbb{R}$
in each entry. Then, each $f_{i,j}$ can repeat a value infinitely
many times if and only if $f_{i,j}$ is a constant. Hence, the same
assertion holds for $f$ and thus for $g$. 
\end{proof}
\begin{cor}
If a coset of an infinite index subgroup $H$ of a finitely generated
nilpotent group $G$ contains infinitely many elements of a fixed
nonzero power of a polynomial sequence $g:\mathbb{N}_{0}\to G$, then
there exists a normal subgroup $N$ of infinite index in $G$ such
that a coset of $N$ contains the whole sequence.
\end{cor}
\begin{proof}
Suppose that a coset of $H$ contains infinitely many of elements
in the sequence $g_{0}^{n},g_{1}^{n},g_{2}^{n},\ldots$ for some $0\ne n\in\mathbb{Z}$.
We may assume that $n>0$.

Since $H$ is of infinite index in $G$, by a technical Lemma \ref{lem:normal subgroup of infinite index}
proved later, there exists a normal subgroup $N$ of infinite index
in $G$ containing $H$. Then, a coset of $N$ contains infinitely
many elements of the sequence $g_{0}^{n},g_{1}^{n},g_{2}^{n},\ldots$,
i.e., the induced sequence $\bar{g}^{n}:\mathbb{N}_{0}\to G\to G/N$
repeats a value infinitely many times.

If $G/N$ is torsion-free, then by Theorem \ref{thm:polynomial sequence repeats constant}
$\bar{g}^{n}$ is a constant. If $G/N$ is not torsion-free, then
the torsion subgroup $\Tor(G/N)$ is a finite normal subgroup of infinite
index in $G/N$ and thus there exists a normal subgroup $N'$ of infinite
index in $G$ such that $G/N'\cong(G/N)/\Tor(G/N)$ is a finitely
generated torsion-free nilpotent group. Since $g^{n}\mod N'=\bar{g}^{n}\mod\Tor(G/N)$
repeats a value infinitely many times, it is constant. So we may assume
that $G/N$ is finitely generated torsion-free nilpotent and $\bar{g}^{n}$
is a constant.

By a classical result on extraction of roots in torsion-free nilpotent
groups (cf. \cite[Thm 2.7]{CMZ2017}), $\bar{g}$ can only take values
from a finite set of at most $n$ elements, which are all $n$th roots
of the constant $\bar{g}^{n}$, and thus must repeat one of the value
infinitely many times. Hence, $\bar{g}$ itself is constant, i.e.,
a coset of $N$ contains the whole sequence $g_{0},g_{1},g_{2},\ldots$.
\end{proof}
Alternatively, we may use the following theorem of P. Hall to embed
a finitely generated torsion-free nilpotent group $G$ into $\mathcal{U}_{n}(\mathbb{Z})$
for some $n$ and to prove Theorem \ref{thm:polynomial sequence repeats constant}. 
\begin{thm*}[{\cite[Thm 7.5]{HallEdmonton1957}, \cite[Thm 6.5]{CMZ2017}}]
\label{thm:faithful repr}  Every finitely generated torsion-free
nilpotent group $G$ is isomorphic to a subgroup of $\mathcal{U}_{n}(\mathbb{Z})$
for some $n=n(G)$. 
\end{thm*}
\begin{rem*}
If one wishes to avoid using Theorem \ref{thm: polynomial sequence in U_n(R)}
and the canonical Mal'tsev embedding or the theorem of Hall to prove
Theorem \ref{thm:polynomial sequence repeats constant}, one could
also argue by induction on the nilpotency class of $G$ via upper
central series $G=Z_{n}\triangleright Z_{n-1}\triangleright\cdots\triangleright Z_{1}\triangleright Z_{0}=\{1\}$.
Without loss of generality, suppose that the polynomial sequence $g$
repeats the value $g_{0}\in G$ infinitely many times. We may assume
that $g_{0}$ is the identity element $1_{G}$ of $G$; otherwise,
replace $g$ by the polynomial sequence $g_{0}^{-1}g$. Let $\pi_{i}:Z_{i}\to Z_{i}/Z_{i-1}$
be the quotient map. Since $G$ is torsion-free, each quotient $Z_{i+1}/Z_{i}$
is torsion-free abelian, cf. \cite[Lem 13.69]{Drutu2018Geometric}
or \cite[Thm 1.2.20]{LennoxRobinson2004}. Consider the induced polynomial
maps $\mathbb{N}_{0}\xrightarrow{g}G\xrightarrow{\pi_{n}}G/Z_{n-1}$.
Then, $G/Z_{n-1}$ is a direct sum of finitely many copies of $\mathbb{Z}$,
since $G$ is finitely generated. By Corollary \ref{cor:vector of polynomial sequences in Z^M, Q^M, R^M},
$\pi_{n}\circ g$ is a vector of polynomials. Since $\pi_{n}\circ g$
vanishes for infinitely many values $n$, it is identically zero,
i.e., $g$ is a polynomial sequence in $Z_{n-1}$, i.e., $g$ has
uc-height $\le n-1$. By induction, one proves that $g$ is a polynomial
sequence in $Z_{i}$, or has uc-height $\le i$, for all $i=n,\ldots,1,0$.
Hence, $g$ is constant. 
\end{rem*}
For the proof given in the above remark, one has to argue with the
upper central series, because for a torsion-free nilpotent group $G$,
the quotients $C^{i}G/C^{i+1}G$ may not be torsion-free. 

Next, we state the technical lemma, whose proof is postponed later. 
\begin{lem}
\label{lem:normal subgroup of infinite index} If $H$ is a subgroup
of infinite index in a finitely generated nilpotent group $G$, there
exists a normal subgroup $N$ of $G$ containing $H$ and having infinite
index in $G$.
\end{lem}
With the help of the above lemma, we can prove the following theorem:
\begin{thm}
\label{thm:infinite subsequence generates finite index subgroup}
Let $G$ be any nilpotent group and $g:\mathbb{N}_{0}\to G$ be any
polynomial sequence such that $G=\langle g\rangle$. Then, every infinite
subsequence (not necessarily corresponding to any arithmetic progression)
generates a finite index subgroup of $G$.
\end{thm}
\begin{proof}
 The assertion is trivial if $G$ is a torsion group.  So we may
assume that $G$ has elements of infinite order. Suppose there exists
an infinite subsequence of $g$, which generates an infinite index
subgroup $H$ of $G$. Then, by Lemma \ref{lem:normal subgroup of infinite index},
there exists a normal subgroup $N$ in $G$ containing $H$ and having
infinite index in $G$. Consider the induced polynomial sequence 
\[
\bar{g}:\mathbb{N}_{0}\xrightarrow{g}G\twoheadrightarrow G/N\twoheadrightarrow(G/N)/\Tor(G/N).
\]
Since $G/N$ is a finitely generated nilpotent group, the torsion
elements of $G/N$ form a finite normal subgroup $\Tor(G/N)$. Since
$G/N$ is infinite, $(G/N)/\Tor(G/N)$ is infinite and thus finitely
generated torsion-free nilpotent. Then, $\bar{g}$ repeat the identity
element in $(G/N)/\Tor(G/N)$ infinitely many times. By Theorem \ref{thm:polynomial sequence repeats constant},
$\bar{g}$ must be the constant identity map. This is a contradiction
to the assumption that $G=\langle g\rangle$, which implies $(G/N)/\Tor(G/N)=\langle\bar{g}\rangle$. 
\end{proof}
 The technique we use to prove Lemma \ref{lem:normal subgroup of infinite index}
is the theory of nearly maximal subgroups. Recall that a proper subgroup
$M$ of a group $G$ is called maximal if it is not properly contained
in any other proper subgroups of $G$. Similarly, a subgroup $M$
of a group $G$ is said to be nearly maximal  if it has infinite
index, but any subgroup properly containing $M$ has finite index
in $G$. 

It is known that every proper subgroup in a finitely generated group
is contained in a maximal one.  Similarly, with the assumption of
Zorn's lemma, we have the following result: 
\begin{lem}
Every proper subgroup of infinite index in a finitely generated infinite
group is contained in a nearly maximal subgroup. 
\end{lem}
\begin{proof}
Let $G$ be any finitely generated infinite group and $H$ be a subgroup
of infinite index. Let 
\[
\Omega_{H}=\{K\le G\mid H\le K,(G:K)=\infty\}
\]
be the partially ordered set of all infinite index subgroups of $G$
containing $H$ with the partial order given by the inclusion of subgroups.
Let $\mathcal{C}$ be a chain (i.e., a totally ordered subset) in
$\Omega_{H}$. Let $J=\bigcup_{K\in\mathcal{C}}K$ be the union of
all elements in $\mathcal{C}$. Then, $J$ is easily seen to be a
subgroup of $G$. By Schreier's lemma,  any finite index subgroup
of a finitely generated group is finitely generated. So if $J$ were
of finite index in $G$, then $J$ would be a finitely generated infinite
group and its finitely many generators would lie in $\bigcup_{K\in\mathcal{C}}K$
and thus in some $K\in\mathcal{C}$, since $\mathcal{C}$ is totally
ordered. Hence, $K=J$ has finite index in $G$, which is a contradiction.
Hence, $J$ must have infinite index in $G$. So far, the hypothesis
of Zorn's lemma has been checked. Then, by Zorn's Lemma, $\Omega_{H}$
contains a maximal element $M$. Then, every subgroup properly containing
$M$ has finite index in $G$. Hence, $M$ is a nearly maximal subgroup
of $G$. 
\end{proof}
Then, Lemma \ref{lem:normal subgroup of infinite index} is an easy
consequence of the following theorem of Lennox and Robison \cite[Thm C]{LennoxRobinson1982}
or \cite[Thm 10.4.5]{LennoxRobinson2004}, which establishes a criterion
for all nearly maximal subgroups being normal in finitely generated
virtually solvable groups. 
\begin{thm*}
Let $G$ be a finitely generated virtually solvable group. Then, the
following conditions are equivalent: 
\begin{enumerate}
\item each nearly maximal subgroup has finitely many conjugates in $G$; 
\item each nearly maximal subgroup is normal in $G$; 
\item $G$ is finite-by-nilpotent.
\end{enumerate}
\end{thm*}
Here, a group $G$ is said to be finite-by-nilpotent, if it has a
normal subgroup $N$ which is finite such that the quotient $G/N$
is nilpotent. 

\section{Symmetric Polynomial Maps \label{sec:Symmetric Polynomial Maps}}

Let $S$ be a commutative semigroup and $S^{N}$ be the direct sum
of $N$ copies of $S$. Let $G$ be any group and denote by $G_{p}^{S^{N}}$
the set of all polynomial maps $S^{N}\to G$. We want to define an
action of the symmetric group $S_{N}$ on the set $G_{p}^{S^{N}}$
by the following manner:

For each $\sigma\in S_{N}$ and each polynomial map 
\[
f:S^{N}\to G;\ (s_{1},s_{2}\ldots,s_{N})\mapsto f(s_{1},s_{2},\ldots,s_{N}),
\]
of degree $d$, we define the function 
\begin{align*}
\sigma(f):S^{N} & \xrightarrow{\sigma}S^{N}\xrightarrow{f}G\\
(s_{1},s_{2}\ldots,s_{N}) & \mapsto(s_{\sigma(1)},s_{\sigma(2)},\ldots,s_{\sigma(N)})\mapsto f(s_{\sigma(1)},s_{\sigma(2)},\ldots,s_{\sigma(N)}).
\end{align*}
Since $\sigma:S^{N}\to S^{N}$ is an isomorphism of commutative semigroups,
by Proposition \ref{prop:homomorphism of commutative semigroups},
$\sigma(f)$ is a polynomial map of degree $d$. Also, we have $e(f)=f$,
where $e\in S_{N}$ is the identity element, and $\sigma\tau(f)=\sigma(\tau(f))$
for all $f\in G_{p}^{S^{N}}$ and $\sigma,\tau\in S_{N}$. Therefore,
this is indeed an action. 
\begin{defn}
A polynomial map $f:S^{N}\rightarrow G$ is called symmetric with
respect to this $S_{N}$-action, if $\sigma(f)=f$ holds for all $\sigma\in S_{N}$. 
\end{defn}
Moreover, if $G_{p}^{S^{N}}$ is a group (for example, when $G$ is
nilpotent of class $n$), then $\sigma$ fixes the identity of $G_{p}^{S^{N}}$,
and $\sigma(fg)=\sigma(f)\sigma(g)$ for all $f,g\in G_{p}^{S^{N}}$
and $\sigma\in S_{N}$. This implies that each $\sigma$ induces a
homomorphism from $S_{N}$ to the automorphism group of $G_{p}^{S^{N}}$.
In this sense, we may say that $G_{p}^{S^{N}}$ is a (left) \textbf{non-abelian}
$S_{N}$-module.  Thus, all symmetric polynomial maps form an $S_{N}$-invariant
subgroup $\left(G_{p}^{S^{N}}\right)^{S_{N}}$ of the group $G_{p}^{S^{N}}$. 

The goal of this section is to prove that every polynomial map from
$S^{N}$ to any nilpotent group $G$ admits an iterated symmetrization.
But let us first start the discussion in a concrete case when $S=\mathbb{R}_{\ge0}$
and $G=\mathcal{U}_{n}(\mathbb{R})$ to illustrate how this is done. 

Any polynomial $f_{i,j}:\mathbb{R}_{\ge0}^{N}\to\mathbb{R}$ admits
a symmetric polynomial $\tilde{f}_{i,j}=\sum_{\sigma\in S_{N}}\sigma(f)$.
This simple fact can be generalized to the following result, which
says that one can symmetrize any continuous polynomial map $f:\mathbb{R}_{\ge0}^{N}\rightarrow\mathcal{U}_{n}(\mathbb{R})$
within a finite number of steps. 
\begin{thm}
\label{thm:symmetrization in U_n(R)} Let $f:\mathbb{R}_{\ge0}^{N}\rightarrow\mathcal{U}_{n}(\mathbb{R})$
be a continuous polynomial map as in Theorem \ref{thm:unitriangular matrix of polynomials in N var}.
Then there is a natural number $M$, only dependent on $N$ and $n$,
and a sequence $\sigma_{1},\sigma_{2},\ldots,\sigma_{M}\in S_{N}$,
such that the product 
\[
\tilde{f}=\prod_{i=1}^{M}\sigma_{i}(f)=\sigma_{1}(f)\sigma_{2}(f)\cdots\sigma_{M}(f)
\]
is a symmetric continuous polynomial map. 
\end{thm}
\begin{proof}
For $n=2$, we have $\mathcal{U}_{2}(\mathbb{R})\cong\mathbb{R}$
and $M$ can be taken to be $N!$. So we may assume that $n>2$. The
proof is given by induction on the $k$th diagonal entries. Clearly,
the first diagonal entries of $f^{\{1\}}:=\prod_{\sigma\in S_{N}}\sigma(f)$
are given by the symmetric polynomials $\prod_{\sigma\in S_{N}}\sigma(f_{i,i+1})$,
$\forall1\le i\le n$. If one sets $M_{1}=N!$, then this gives the
basis step. 

Suppose that there is a finite sequence $\sigma_{1},\sigma_{2},\ldots,\sigma_{M_{k-1}}$
for some $M_{k-1}$, such that the first, $\ldots$, $(k-1)$th diagonal
entries of $f^{\{k-1\}}:=\prod_{i=1}^{M_{k-1}}\sigma_{i}(f)$ are
all symmetric polynomials. The goal is to show that the first, $\ldots$,
$k$th diagonal entries of $f^{\{k\}}:=\prod_{\sigma\in S_{N}}\sigma(f^{\{k-1\}})$
are all symmetric polynomials. Clearly, the first, $\ldots$, $(k-1)$th
diagonal entries of $f^{\{k\}}:=\prod_{\sigma\in S_{N}}\sigma(f^{\{k-1\}})$
remain symmetric, because they are linear combinations of finite products
of first, $\ldots$, $(k-1)$th diagonal entries of $f^{\{k-1\}}$,
which are symmetric by induction. Similarly, the $k$th diagonal entries
of $f^{\{k\}}$ are given by the summation of 
\[
\sum_{\sigma\in S_{N}}\sigma(f_{i,i+k}^{\{k-1\}})
\]
and linear combinations of finite products of first, $\ldots$, $(k-1)$th
diagonal entries of $f^{\{k-1\}}$, both of which are symmetric. This
proves the inductive step. The induction method implies that $M$
can be taken to be $(N!)^{n-1}$ and hence the proof is complete. 
\end{proof}
 The idea given in previous proof suggests the following more general
result:
\begin{thm}
\label{thm:symmetrization in G} Let $S$ be a commutative semigroup,
$G$ be a nilpotent group of class $n$ and $f:S^{N}\rightarrow G$
be a polynomial map. Then there is a natural number $M$, only dependent
on $N$ and $n$, and a sequence $\sigma_{1},\sigma_{2},\ldots,\sigma_{M}\in S_{N}$,
such that the product 
\[
\tilde{f}=\prod_{i=1}^{M}\sigma_{i}(f)=\sigma_{1}(f)\sigma_{2}(f)\cdots\sigma_{M}(f):S^{N}\rightarrow G
\]
is a symmetric polynomial map.

Moreover, if the group $\langle f\rangle$ generated by $f(S^{N})$
is finitely generated and the subgroup $\langle f\restriction_{S}\rangle$
generated by the image of the restriction of $f$ on the diagonal
$S$ of $S^{N}$ has finite index in $\langle f\rangle$, then the
subgroup $\langle\tilde{f}\rangle$ generated by $\tilde{f}(S^{N})$
is of finite index in $\langle f\rangle$.
\end{thm}
\begin{proof}
If $n=0$, then the theorem is trivial. If $n=1$, then setting $M=N!$
and $\{\sigma_{1},\sigma_{2},\ldots,\sigma_{M}\}=S_{N}$, one finds
that the product 
\[
\tilde{f}=\prod_{i=1}^{M}\sigma_{i}(f)=\prod_{\sigma\in S_{N}}\sigma(f)
\]
satisfies that $\tau(\tilde{f})=\tilde{f}$ for all $\tau\in S_{N}$,
since $G$ is abelian. Thus, $\tilde{f}$ is a symmetric polynomial
map. 

So we may assume that $n\ge2$. If one defines $f_{1}=\prod_{\sigma\in S_{N}}\sigma(f)$,
then for any $\tau\in S_{N}$, one has 
\[
\tau(f_{1})=\prod_{\sigma\in S_{N}}\tau\sigma(f)\equiv f_{1}\mod C^{2}G,
\]
i.e., $f_{1}$ is symmetric modulo $C^{2}G$. Write $\tau(f_{1})=f_{1}\alpha_{\tau}$,
where $\alpha_{\tau}=f_{1}^{-1}\tau(f_{1})$ is a polynomial map from
$S^{N}$ to $C^{2}G$. Then, $\alpha_{e}=f_{1}^{-1}f_{1}$ is the
identity of $G_{p}^{S^{N}}$, where $e\in S_{N}$ is the identity,
and 
\[
f_{1}\alpha_{\sigma\tau}=\sigma\tau(f_{1})=\sigma(f_{1}\alpha_{\tau})=\sigma(f_{1})\sigma(\alpha_{\tau})=f_{1}\alpha_{\sigma}\sigma(\alpha_{\tau})
\]
and thus $\alpha_{\sigma\tau}=\alpha_{\sigma}\sigma(\alpha_{\tau})$.
Then, we obtain a $1$-cocycle $\alpha:S_{N}\to C^{2}(G_{p}^{S^{N}})$. 

If one defines $f_{2}=\prod_{\sigma\in S^{N}}\sigma(f_{1})$, then
for any $e\ne\tau\in S_{N}$, one has 
\begin{align*}
\tau(f_{2}) & =\prod_{\sigma\in S^{N}}\tau\sigma(f_{1})=\prod_{\sigma\in S^{N}}\sigma\sigma^{-1}\tau\sigma(f_{1})\\
 & =\prod_{\sigma\in S^{N}}\sigma(f_{1}\alpha_{\sigma^{-1}\tau\sigma})=\prod_{\sigma\in S^{N}}\sigma(f_{1})\sigma(\alpha_{\sigma^{-1}\tau\sigma}).
\end{align*}
Notice that $\sigma(\alpha_{\sigma^{-1}\tau\sigma}):S^{N}\to C^{2}G$
and $\sigma(f_{1}):S^{N}\to G$ are polynomial maps for all $\sigma\in S^{N}$
and commutators of such terms are certainly polynomial maps from $S^{N}$
to $C^{3}G$. Pushing $\sigma(\alpha_{\sigma^{-1}\tau\sigma})$ to
the rightmost, we see that 
\[
\tau(f_{2})\equiv f_{2}\mod C^{3}G,
\]
i.e., $f_{2}$ is symmetric modulo $C^{3}G$. Similarly, we write
$\tau(f_{2})=f_{2}\beta_{\tau}$, where $\beta_{\tau}=f_{2}^{-1}\tau(f_{2})$
is a polynomial map from $S^{N}$ to $C^{3}G$. Then, $\beta_{e}=f_{2}^{-1}f_{2}$
is the constant map to the identity of $G$, and 
\[
f_{2}\beta_{\sigma\tau}=\sigma\tau(f_{2})=\sigma(f_{2}\beta_{\tau})=\sigma(f_{2})\sigma(\beta_{\tau})=f_{2}\beta_{\sigma}\sigma(\beta_{\tau})
\]
and thus $\beta_{\sigma\tau}=\beta_{\sigma}\sigma(\beta_{\tau})$.
Then, we obtain another $1$-cocycle $\beta:S_{N}\to C^{3}(G_{p}^{S^{N}})$. 

It is clear that the procedure described above continues and ends
in finitely many steps, since $G_{p}^{S^{N}}$ is nilpotent of class
$n$. Hence, there is a finite sequence $\sigma_{1},\sigma_{2},\ldots,\sigma_{M}\in S_{N}$,
where $M=(N!)^{n}$ such that the  product 
\[
\tilde{f}=\prod_{i=1}^{M}\sigma_{i}(f)=\sigma_{1}(f)\sigma_{2}(f)\cdots\sigma_{M}(f):S^{N}\to G
\]
is a symmetric polynomial map.

We have the following relations of inclusion of subgroups in $\langle f\rangle$:
\[
\langle(f\restriction_{S})^{M}\rangle\subset\langle f\restriction_{S}\rangle\subset\langle f\rangle\text{ and }\langle(f\restriction_{S})^{M}\rangle=\text{\ensuremath{\langle\tilde{f}\restriction_{S}\rangle}}\subset\langle\tilde{f}\rangle\subset\langle f\rangle.
\]
Since $\langle f\rangle$ is finitely generated and nilpotent, $\langle f\restriction_{S}\rangle$
is finitely generated and nilpotent. By a result due to Mal'tsev (cf.
\cite[Thm 2.23]{CMZ2017}), $\langle(f\restriction_{S})^{M}\rangle$
has finite index in $\langle f\restriction_{S}\rangle$ and since
$\langle f\restriction_{S}\rangle$ has finite index in $\langle f\rangle$,
it also has finite index in $\langle f\rangle$. Hence, $\langle\tilde{f}\rangle$
has finite index in $\langle f\rangle$.
\end{proof}
\begin{rem*}
The strategy given in proofs above will be called iterated symmetrization,
which turns out to be very useful in the sequel. 
\end{rem*}

\section{Polynomial Sets \label{sec:Polynomial Sets}}
\begin{defn}
A subset $U$ of a \emph{ path-connected} nilpotent Lie group $N$
is said to be parameterized by some \emph{continuous} polynomial map
$f:\mathbb{R}_{\ge0}^{n}\to N$, if it is the image of $f$. In this
case, we denote $(U\mid f:\mathbb{R}_{\ge0}^{n}\to N)$ and call it
a polynomial set in $N$; but sometimes we will abbreviate $f$ and
simply call $U$ a polynomial set in $N$ for short. If $U$ is open
(resp. is closed, resp. has nonempty interior) in  $N$, then we
call $U$ an open (resp. a closed, resp. a proper) polynomial set
in $N$.

A nonempty subset $V$ of a nilpotent group $G$ is called a polynomial
set, if it is the inverse image $\phi^{-1}(U)$ of a polynomial set
$(U\mid f:\mathbb{R}_{\ge0}^{n}\to N)$ of a nilpotent Lie group $N$
along some group homomorphism $\phi\colon G\to N$. In this case,
for completeness, we denote the polynomial set by 
\[
(V\mid\phi:G\to N,(U\mid f:\mathbb{R}_{\ge0}^{n}\to N)).
\]
We call $V=\phi^{-1}(U)$ an open (resp. a closed, resp. a proper)
polynomial set in $G$, if $U$ has the same property in $N$. In
particular, the generalized cones in $G$ is given by proper polynomial
set. 
\end{defn}
\begin{rem}
One may wonder why we require the nilpotent Lie group $N$ to be path-connected.
For one reason, continuous image of a path-connected set is path-connected;
for the other, if we do not assume this, then in the trivial sense
any group having at most countably many elements can be viewed as
a $0$-dimensional Lie group with the discrete topology. For example,
the group $\mathcal{U}_{n}(\mathbb{Z})$ of upper unitriangular $n\times n$
matrices in integers, has at most countably many elements. Then, any
singleton of $\mathcal{U}_{n}(\mathbb{Z})$ viewed as a $0$-dimensional
Lie group with the discrete topology is open. 
\end{rem}
\begin{rem*}
Corollary \ref{cor:polynomial set in N var} implies that the only
polynomial sets in $\mathbb{R}$ are singletons, unbounded closed
intervals and the whole $\mathbb{R}$.  Hence, the only open polynomial
set in $\mathbb{R}$ is $\mathbb{R}$ itself, and the only proper
polynomial sets in $\mathbb{R}$ are either unbounded closed intervals
or the whole $\mathbb{R}$. 
\end{rem*}
The following lemma proves the existence of a proper polynomial set
inside any Kamke domains. 
\begin{lem}
\label{lem:polynomial set in Kamke's domain} For any $B\ge2$, a
Kamke domain 
\[
U(B,N)=\{(l_{1},\cdots,l_{B})\in\mathbb{R}_{\ge0}^{B}\mid k_{1}<l_{1},k_{\nu}l_{1}^{v}<l_{v}<K_{v}l_{1}^{v},\nu=2,3,\ldots,B\}
\]
always contains a proper polynomial set. 
\end{lem}
\begin{proof}
Let $y_{\kappa}=x_{\kappa}+k_{1}/n+\varepsilon/n$ and consider the
following polynomial map 
\[
q:\mathbb{R}_{\ge0}^{n}\to\mathbb{R}^{B};\ (x_{1},\ldots,x_{n})\mapsto(l_{1},\cdots,l_{B}),
\]
where  
\[
\left\{ \begin{aligned}l_{1} & =\sum_{\kappa=1}^{n}y_{\kappa}>k_{1},\\
l_{2} & =C_{2}\sum_{\kappa=1}^{n}y_{\kappa}^{2}+D_{2}l_{1}^{2}+\varepsilon,\\
\cdots & \cdots\cdots\cdots\\
l_{B} & =C_{B}\sum_{\kappa=1}^{n}y_{\kappa}^{B}+D_{B}l_{1}^{B}+\varepsilon,
\end{aligned}
\right.
\]
and $\varepsilon,C_{2},\ldots,C_{B},D_{2},\ldots,D_{B}$ are some
positive numbers. By the multinomial formula, 
\begin{align*}
l_{1}^{\nu} & =\left(\sum_{\kappa=1}^{n}y_{\kappa}\right)^{\nu}=\sum_{j_{1}+j_{2}+\cdots+j_{n}=\nu}\binom{\nu}{j_{1},j_{2},\ldots,j_{n}}\prod_{\kappa=1}^{n}y_{\kappa}^{j_{\kappa}}\\
 & =\sum_{\kappa=1}^{n}y_{\kappa}^{\nu}+\sum_{\substack{j_{1}+j_{2}+\cdots+j_{n}=\nu\\
j_{1},j_{2},\ldots,j_{n}\ne\nu
}
}\binom{\nu}{j_{1},j_{2},\ldots,j_{n}}\prod_{\kappa=1}^{n}y_{\kappa}^{j_{\kappa}}\\
 & \ge\sum_{\kappa=1}^{n}y_{\kappa}^{\nu}+\sum_{\substack{j_{1}+j_{2}+\cdots+j_{n}=\nu\\
j_{1},j_{2},\ldots,j_{n}\ne\nu
}
}\binom{\nu}{j_{1},j_{2},\ldots,j_{n}}\prod_{\kappa=1}^{n}\left(k_{1}/n+\varepsilon/n\right)^{j_{\kappa}}\\
 & =\sum_{\kappa=1}^{n}y_{\kappa}^{\nu}+(n^{\nu}-n)\left(k_{1}/n+\varepsilon/n\right)^{\nu}\\
 & =\sum_{\kappa=1}^{n}y_{\kappa}^{\nu}+(1-n^{1-\nu})(k_{1}+\varepsilon)^{\nu},
\end{align*}
where equality holds if all $x_{\kappa}$ are $0$, and by the generalized
mean inequality, 
\[
\left(\dfrac{l_{1}}{n}\right)^{\nu}=\left(\frac{1}{n}\sum_{\kappa=1}^{n}y_{\kappa}\right)^{\nu}\le\frac{1}{n}\sum_{\kappa=1}^{n}y_{\kappa}^{\nu},
\]
where equality holds if all $x_{\kappa}$ are the same. Then, for
all $\nu=2,3,\ldots,B$, we must have 
\begin{align*}
(C_{\nu}n^{1-\nu}+D_{\nu})l_{1}^{\nu}+\varepsilon & \le C_{\nu}\sum_{\kappa=1}^{n}y_{\kappa}^{\nu}+D_{\nu}l_{1}^{\nu}+\varepsilon=l_{\nu}\\
 & \le C_{\nu}\left(l_{1}^{\nu}-(1-n^{1-\nu})(k_{1}+\varepsilon)^{\nu}\right)+D_{\nu}l_{1}^{\nu}+\varepsilon\\
 & =(C_{\nu}+D_{\nu})l_{1}^{\nu}-(C_{\nu}(1-n^{1-\nu})(k_{1}+\varepsilon)^{\nu}-\varepsilon).
\end{align*}
For $k_{\nu}l_{1}^{v}<l_{v}<K_{v}l_{1}^{v}$ to be true, it suffices
to take some sufficiently large $n>1$, $C_{\nu},D_{\nu}$ satisfying
\[
0<k_{\nu}-C_{\nu}n^{1-\nu}=D_{\nu}=K_{\nu}-C_{\nu},\qquad0<C_{\nu}=\dfrac{K_{\nu}-k_{\nu}}{1-n^{1-\nu}},
\]
and some sufficiently small $\varepsilon>0$ such that 
\begin{align*}
C_{\nu}(1-n^{1-\nu})(k_{1}+\varepsilon)^{\nu}-\varepsilon & =\dfrac{K_{\nu}-k_{\nu}}{1-n^{1-\nu}}(1-n^{1-\nu})(k_{1}+\varepsilon)^{\nu}-\varepsilon\\
 & =(K_{\nu}-k_{\nu})(k_{1}+\varepsilon)^{\nu}-\varepsilon\\
 & >(K_{\nu}-k_{\nu})k_{1}^{\nu}-\varepsilon>0.
\end{align*}
To show that $q(\mathbb{R}_{\ge0}^{n})$ has a nonempty interior inside
$U(B,N)$, one check that the rank of the Jacobian matrix of $q$
is $B$, which basically follows from the linearly independence of
geometric series. 
\end{proof}
\printindex
\bibliographystyle{alpha}

\begin{thebibliography}{99}

\bibitem[CMZ17]{CMZ2017}
A.~E. Clement, S.~Majewicz, and M.~Zyman.
\newblock {\em The Theory of Nilpotent Groups}.
\newblock Birkh{\"a}user Basel, 2017.

\bibitem[DK18]{Drutu2018Geometric}
C.~Dru{\c{t}}u and M.~Kapovich.
\newblock {\em Geometric Group Theory}.
\newblock Colloquium Publications. American Mathematical Society, 2018.

\bibitem[Hal57]{HallEdmonton1957}
P.~G. Hall.
\newblock {\em The {E}dmonton notes on nilpotent groups}.
\newblock Queen Mary College Mathematics Notes. Queen Mary College Department
  of Mathematics, 1957.

\bibitem[Kam21]{Kamke1921}
E.~Kamke.
\newblock Verallgemeinerungen des {W}aring-{H}ilbertschen {S}atzes.
\newblock {\em Mathematische Annalen}, 83(1):85--112, Mar 1921.

\bibitem[Kuc09]{Kuczma2009Functional}
M.~Kuczma.
\newblock {\em An Introduction to the Theory of Functional Equations and
  Inequalities: Cauchy's Equation and Jensen's Inequality}.
\newblock Birkh{\"a}user Basel, 2nd edition, 2009.

\bibitem[Lei98]{Leibman1998}
A~Leibman.
\newblock Polynomial sequences in groups.
\newblock {\em Journal of Algebra}, 201(1):189--206, 1998.

\bibitem[Lei02]{Leibman2002}
A.~Leibman.
\newblock Polynomial mappings of groups.
\newblock {\em Israel Journal of Mathematics}, 129(1):29--60, Dec 2002.

\bibitem[LR82]{LennoxRobinson1982}
John~C. Lennox and Derek J.~S. Robinson.
\newblock Nearly maximal subgroups of finitely generated soluble subgroups.
\newblock {\em Archiv der Mathematik}, 38(1):289--295, Dec 1982.

\bibitem[LR04]{LennoxRobinson2004}
John~C. Lennox and Derek J.~S. Robinson.
\newblock {\em The Theory of Infinite Soluble Groups}.
\newblock Oxford Mathematical Monographs. Oxford: Clarendon Press; New York:
  Oxford University Press, 2004.

\bibitem[Mal49]{Mal49b}
A.~I. Mal'tsev.
\newblock On a class of homogeneous spaces.
\newblock {\em Izvestiya Akad. Nauk. SSSR Ser. Mat.}, 13(1):9--32, 1949.

\end{thebibliography}

\end{document}